\documentclass[11pt]{article}
\usepackage{authblk}
\usepackage{amsmath,amssymb,mathrsfs}
\usepackage{enumitem}
\usepackage{verbatim}
\usepackage{color}
\usepackage[margin=1in]{geometry}
\usepackage{setspace}
\usepackage{graphics,graphicx}
\usepackage{pstricks,pst-node,pst-tree}
\usepackage{tikz-cd}

\newcommand{\lang}{\mathcal{L}}
\newcommand{\limp}{\longrightarrow}
\newcommand{\qneg}{\sim\!}
\newcommand{\pow}{\mathcal{P}}

\newcommand{\var}{\mathrm{var}}
\newcommand{\NN}{\mathbb{N}}
\newcommand{\ZZ}{\mathbb{Z}}

\usepackage{amsthm}
\theoremstyle{definition}
\newtheorem{thm}{Theorem}[section]
\newtheorem{lem}[thm]{Lemma}
\newtheorem{cor}[thm]{Corollary}
\newtheorem{dfn}[thm]{Definition}
\newtheorem{rem}[thm]{Remark}
\newtheorem{exa}[thm]{Example}

\usepackage[backref=page,colorlinks=true,allcolors=blue]{hyperref}

  \title{Generalized explosion principles\footnote{The final version of this article has been submitted to Studia Logica.}}
  \author[1]{Sankha S. Basu}
  \author[1]{Sayantan Roy}
  \date{May 27, 2024}
  \affil[1]{Department of Mathematics\\
  Indraprastha Institute of Information Technology-Delhi\\
  New Delhi, India.}

\begin{document}
\maketitle

\begin{abstract}
    \noindent Paraconsistency is commonly defined and/or characterized as the failure of a principle of explosion. The various standard forms of explosion involve one or more logical operators or connectives, among which the negation operator is the most frequent and primary. In this article, we start by asking whether a negation operator is essential for describing explosion and paraconsistency. In other words, is it possible to describe a principle of explosion and hence a notion of paraconsistency that is independent of connectives? A negation-free paraconsistency resulting from the failure of a generalized principle of explosion is presented first. We also derive a notion of \emph{quasi-negation} from this and investigate its properties. Next, more general principles of explosion are considered. These are also negation-free; moreover, these principles gradually move away from the idea that an explosion requires a statement and its opposite. Thus, these principles can capture the explosion observed in logics where a statement and its negation explode only in the presence of additional information, such as in the logics of formal inconsistency.
\end{abstract}

\tableofcontents

\section{Introduction}
A principle of explosion can be described broadly as a rule that allows one to entail every $\beta\in\lang$, where $\langle\lang,\vdash\rangle$ is a logic, from a $\Gamma\subsetneq\lang$, i.e., $\Gamma\vdash\beta$ for all $\beta\in\lang$. The set $\Gamma$ is then said to be \emph{trivial} or to \emph{explode}. Perhaps the most discussed principle of explosion is the \emph{ex contradictione sequitur quodlibet (ECQ)} which asserts that any $\Gamma\subseteq\lang$ that contains, for some $\alpha\in\lang$, the pair $\alpha,\neg\alpha$, where $\neg$ is a negation operator, explodes. Assuming the Tarskian monotonicity condition, this is then expressed symbolically as follows. For any $\alpha\in\lang$,
\[
\{\alpha,\neg\alpha\}\vdash\beta\quad\hbox{ for all }\beta\in\lang.
\]
However, principles of explosion come in various other forms, as discussed in \cite{Robles2009}. These different forms of explosion are all equivalent in classical logic but can differ in case of weaker logics. 

Paraconsistency is characterized by the failure of a principle of explosion. Thus, failure of different principles of explosion lead to different notions of paraconsistency. Hence it is possible to have logics paraconsistent in one sense but not in the other. We have discussed this phenomenon in more detail in \cite{BasuRoy2022}. An observation we made there is the dependence of all but one principle of explosion on a negation operator. This is not surprising considering that a contradiction can be described as the simultaneous assertion and denial of a statement, the latter of which is often regarded as equivalent to accepting the negation of the statement. However, it is not hard to see that the denial of a statement may not amount to the acceptance of its negation and that a statement can oppose another without being the negation of it (see, e.g., the discussion on this in \cite[Chapter 1]{CarnielliConiglio2016}).

There is, however, one principle of explosion, and hence a corresponding notion of paraconsistency, that is slightly different as it does not explicitly depend on negation. Instead, it depends on the presence of a nullary operator or constant and its properties. Suppose $\mathcal{S}=\langle\lang,\vdash\rangle$ is a logic and $\bot\in\lang$ is a nullary operator or constant such that $\{\bot\}\vdash\beta$ for all $\beta\in\lang$. This is a principle of explosion that we can call $\bot$-ECQ\footnote{This is also referred to as the principle of \emph{ex falso sequitur quodlibet}, as in \cite{CarnielliConiglio2016, Carnielli2007}.}, and such a constant, a \emph{falsity constant} for the logic $\mathcal{S}$. A logic $\mathcal{S}=\langle\lang,\vdash\rangle$ is said to be \emph{$\bot$-paraconsistent} if there is no falsity constant in $\lang$ for the logic.

Now, as discussed in \cite{Robles2009}, a falsity constant $\bot$ (referred to as ``Das Absurde'' there) can be understood in multiple ways. Under certain understandings, $\bot$-paraconsistency is essentially the same as paraconsistency, while in others, they fork. Moreover, there are understandings of $\bot$ that make the notion of $\bot$-paraconsistency inherently dependent on negation, while in others where $\bot$ is not definitionally eliminable (see \cite{RoblesMendez2008}), that is not the case. In the latter case, although this notion of paraconsistency is different from that of paraconsistency and independent of negation, it requires the presence of another logical operator, viz., $\bot$.

The broad question we dealt with in \cite{BasuRoy2022} is the following. Is it possible to describe paraconsistency (and thus, also explosion) independently of the language or connectives, especially negation? One motivation for this question comes from universal logic \cite{Beziau1994,Beziau2006}. The fundamental notion in universal logic is that of a logical structure, which is a pair of the form $\langle\lang,\vdash\rangle$, where $\lang$ is a set and $\vdash\,\subseteq\pow(\lang)\times\lang$. A logical structure is thus defined without regard to connectives, and the set $\lang$ need not be a formula algebra as in a logic. The above question can therefore be rephrased as follows. Can paraconsistency be defined in the context of universal logic, i.e., is it possible to define something like a \emph{paraconsistent logical structure}? We do not focus on the distinction between a logic and a logical structure here and use the former term throughout the article at the cost of a slight abuse of terminology. However, unless explicitly stated, $\lang$ is not assumed to be a formula algebra or even an algebra. In some of these special cases, where $\mathcal{S}=\langle\lang,\vdash\rangle$ is a logic such that $\lang$ is an algebra, and $\neg$ is a unary operator in $\mathcal{S}$, with respect to which we discuss explosion or the failure of it, one might think of these as negation operators. Keeping in mind that there might be more than one such unary operator in $\mathcal{S}$, we will refer to ECQ written in terms of $\neg$ as $\neg$-ECQ, dropping the prefix only when there is no chance of confusion or when commenting in general about $\neg$-ECQ for some unary operator $\neg$. If $\neg$-ECQ fails in $\mathcal{S}$, then the logic $\mathcal{S}$ will be called paraconsistent with respect to $\neg$. 

We have freely moved back and forth between the relation $\vdash$ and its corresponding operator $C_\vdash:\pow(\lang)\to\pow(\lang)$. We outline the details below.

Let $\mathcal{S}=\langle\lang,\vdash\rangle$ be a logic and $\Gamma\subseteq\lang$. Then $C_{\vdash}^{\mathcal{S}}(\Gamma)$ denotes the set of all elements of $\lang$ that are $\vdash$-related to $\Gamma$, i.e., $C_{\vdash}^{\mathcal{S}}(\Gamma)=\{\alpha\in \lang:\Gamma\vdash \alpha\}$ (the superscript $\mathcal{S}$ is dropped when there is no chance of confusion). It is easy to see that given a $\vdash$, $C_\vdash$, as described above, is an operator from $\pow(\lang)$ to $\pow(\lang)$. Conversely, given an operator $C:\pow(\lang)\to\pow(\lang)$, we can define a relation $\vdash\,\subseteq\pow(\lang)\times\lang$, such that $C=C_\vdash$, as follows. For all $\Gamma\cup\{\alpha\}\subseteq\lang$, $\Gamma\vdash\alpha$ iff $\alpha\in C(\Gamma)$. We will, on some occasions in this article, define a $\vdash\,\subseteq\pow(\lang)\times\lang$ such that $C_\vdash=C$ for some operator $C:\pow(\lang)\to\pow(\lang)$ in this very sense. 

The following notational convention has been adopted in this article. Given a logic $\langle\lang,\vdash\rangle$, for any $\alpha,\beta\in\lang$, if $\{\alpha\}\vdash\beta$, we will write $\alpha\vdash\beta$ instead. In the same vein, for any $\alpha\in\lang$, we will write $C_\vdash(\alpha)$ instead of $C_\vdash(\{\alpha\})$.

Section \ref{sec:nfp} discusses and extends one of the answers provided in \cite{BasuRoy2022} to the broad question mentioned above. This is achieved via a generalization of ECQ. Here we include the discussion on a \emph{quasi-negation}, also introduced in \cite{BasuRoy2022}. This, in a way, points to the extent to which negation is fundamental to the notion of paraconsistency.

In the remainder of the paper, our focus rests primarily on the notion of explosion. The main question thus becomes the following. What is a principle of explosion? We discuss various possible alternatives to the already-known principles of explosion and their interconnections. Thus, in a way, this generalizes and extends the discussion started in \cite{BasuRoy2022}. However, we neither claim nor believe that the list of principles of explosion provided here is exhaustive.

\section{Negation-free paraconsistency}\label{sec:nfp}
The material in this section is an extension of \cite[Section 2]{BasuRoy2022}. We have included the results from there, without proof, along with the new developments, for context and continuity.

The standard principle of explosion, ECQ, can be generalized as follows. For each $\alpha\in\lang$ in a logic $\mathcal{S}=\langle\lang,\vdash\rangle$, there exists $\beta\in\lang$ such that 
\[
\{\alpha,\beta\}\vdash\gamma, \hbox{ where }\gamma\in\lang \hbox{ is arbitrary}.
\]
We call this rule gECQ (the letter `g' denoting `generalized'). In other words, gECQ says that for each $\alpha\in\lang$, there exists $\beta\in\lang$ such that $C_\vdash(\{\alpha,\beta\})=\lang$. Then a type of paraconsistency can be defined without referring to a negation operator as follows.

\begin{dfn}
A logic $\mathcal{S}=\langle\lang,\vdash\rangle$ is \emph{NF-paraconsistent} (NF stands for negation-free) if gECQ fails in it, i.e., there exists an $\alpha\in\lang$ such that for every $\beta\in\lang$, $C_\vdash(\{\alpha,\beta\})\neq\lang$.
\end{dfn}

\begin{thm}\label{thm:NFpara->para}\cite[Theorem 2.2]{BasuRoy2022}
Suppose $\mathcal{S}=\langle\lang,\vdash\rangle$ is an NF-paraconsistent logic with a unary operator $\neg$. Then $\mathcal{S}$ is also paraconsistent with respect to $\neg$.
\end{thm}

It is, however, possible for a logic to be paraconsistent but not NF-paraconsistent (see Example \ref{exa:PWK-NotNFpara} below). Before proceeding any further, we recall the well-known Tarskian conditions below, as these will repeatedly come up in our discussions. Suppose $\mathcal{S}=\langle\lang,\vdash\rangle$ is a logic.
\begin{enumerate}[label=(\roman*)]
    \item For all $\Gamma\cup\{\alpha\}\subseteq\lang$, if $\alpha\in\Gamma$, then $\Gamma\vdash\alpha$. In other words, for all $\Gamma\subseteq\lang$, $\Gamma\subseteq C_\vdash(\Gamma)$. (\emph{Reflexivity})
    \item For all $\Gamma,\Sigma\subseteq\lang$ and $\alpha\in\lang$, if $\Gamma\vdash\alpha$ and $\Gamma\subseteq\Sigma$, then $\Sigma\vdash\alpha$. Equivalently, in terms of $C_\vdash$, for all $\Gamma,\Sigma\subseteq\lang$, if $\Gamma\subseteq\Sigma$, then $C_\vdash(\Gamma)\subseteq C_\vdash(\Sigma)$. (\emph{Monotonicity})
    \item For all $\Gamma,\Sigma\subseteq\lang$ and $\alpha\in\lang$, if $\Gamma\vdash\beta$ for all $\beta\in\Sigma$ and $\Sigma\vdash\alpha$, then $\Gamma\vdash\alpha$. In other words, for all $\Gamma,\Sigma\subseteq\lang$, if $\Sigma\subseteq C_\vdash(\Gamma)$, then $C_\vdash(\Sigma)\subseteq C_\vdash(\Gamma)$. (\emph{Transitivity})
\end{enumerate}
The logic $\mathcal{S}$ (and also the relation $\vdash$) is said to be reflexive/monotonic/transitive when the above reflexivity/monotonicity/transitivity condition is satisfied, respectively. $\mathcal{S}$ is said to be \emph{Tarskian} if all three conditions are satisfied simultaneously and \emph{non-Tarskian} if one or more of these fail.

\begin{thm}\label{thm:NFpara->Fpara}
Suppose $\mathcal{S}=\langle\lang,\vdash\rangle$ is a monotonic logic. If $\mathcal{S}$ is NF-paraconsistent, then there does not exist any $\varphi\in\lang$ such that $C_\vdash(\varphi)=\lang$.
\end{thm}

\begin{proof}
Suppose there exists $\varphi\in\lang$ such that $C_\vdash(\varphi)=\lang$. Now, let $\alpha\in\lang$. Then $C_\vdash(\{\alpha,\varphi\})=\lang$ by monotonicity. Thus, for every $\alpha\in\lang$, there exists $\beta\in\lang$, viz., $\varphi$, such that $C_\vdash(\{\alpha,\beta\})=\lang$. Hence $\langle\lang,\vdash\rangle$ is not NF-paraconsistent. Thus, if $\mathcal{S}$ is a monotonic NF-paraconsistent logic, then there cannot be a $\varphi\in\lang$ such that $C_\vdash(\varphi)=\lang$.
\end{proof}

\begin{rem}
It is well-known that the Tarskian reflexivity and transitivity conditions together imply monotonicity. Thus, the above result also holds for logics that are simultaneously reflexive and transitive.
\end{rem}

We now obtain the following result of \cite{BasuRoy2022} as a corollary to the above theorem.

\begin{cor}\cite[Theorem 2.3]{BasuRoy2022}\label{cor:NFpara->Fpara}
    Suppose $\mathcal{S}=\langle\lang,\vdash\rangle$ is a monotonic logic. If there is a falsity constant for the logic in $\lang$, then $\mathcal{S}$ is not NF-paraconsistent, i.e., if $\mathcal{S}$ is NF-paraconsistent, then it must be $\bot$-paraconsistent.
\end{cor}

\begin{proof}
    Clear, since if there is a falsity constant for the logic $\bot\in\lang$, then $C_\vdash(\bot)=\lang$.
\end{proof}

We now turn to some examples of NF-paraconsistent logics.

\begin{exa}\cite[Example 2.7]{BasuRoy2022}
Suppose $\mathcal{S}=\langle\lang,\vdash\rangle$ is a logic, where $\vdash$ is defined as follows. For any $\Gamma\cup\{\alpha\}\subseteq\lang$,
$\Gamma\vdash\alpha$ iff $\alpha\in\Gamma$. Such a logic, which may be called a \emph{purely reflexive} logic, is clearly NF-paraconsistent as long as $\lang$ has at least three distinct elements.

Such a purely reflexive logic is also monotone and transitive, and hence Tarskian.
\end{exa}

\begin{exa}\label{exa:NFpara3}
Suppose $(\lang,\le)$ is a poset with two incomparable elements $\alpha,\beta$ (i.e., $\alpha\not\le\beta$ and $\beta\not\le\alpha$). We now define a logic $\mathcal{S}=\langle\lang,\vdash\rangle$ as follows. For any $\Gamma\cup\{\varphi\}\subseteq\lang$, $\Gamma\vdash\varphi$ iff for all $\gamma\in\Gamma$, $\gamma\le\varphi$. Clearly, for all $\gamma\in\lang$, $\{\alpha,\gamma\}\not\vdash\beta$ and $\{\beta,\gamma\}\not\vdash\alpha$. Thus, for all $\gamma\in\lang$, $C_\vdash(\{\alpha,\gamma\})\neq\lang$ and $C_\vdash(\{\beta,\gamma\})\neq\lang$. Hence $\mathcal{S}$ is NF-paraconsistent.

We can define a dual logic, $\mathcal{S}^\prime=\langle\lang,\vdash^\prime\rangle$, of $\mathcal{S}$ as follows. For any $\Gamma\cup\{\varphi\}\subseteq\lang$, $\Gamma\vdash^\prime\varphi$ iff for all $\gamma\in\Gamma$, $\varphi\le\gamma$. $\mathcal{S}^\prime$ is also clearly NF-paraconsistent for the same reasons as above.

The logics described above are all non-reflexive and non-monotonic, and hence non-Tarskian.
\end{exa}

The above example is mentioned in \cite[Remark 2.10]{BasuRoy2022}. See also \cite[Examples 2.8, 2.9]{BasuRoy2022} for special cases of the general ideas described in Example \ref{exa:NFpara3}.

Continuing along the same lines, let $\lang$ be any set, $(P\le)$ be any poset, and $\emptyset\subsetneq\mathcal{V}\subseteq P^\lang$ be a set of functions (these may be thought of as \emph{valuation} functions). We can now define a relation $\models_\mathcal{V}\,\subseteq\pow(\lang)\times\lang$ (similar to a \emph{semantic consequence} relation) as follows. For all $\Gamma\cup\{\alpha\}\subseteq\lang$,
\[
\Gamma\models_\mathcal{V}\alpha\quad \hbox{ iff }\quad\hbox{ for all }v\in\mathcal{V},\, v(\beta)\le v(\alpha)\hbox{ for every }\beta\in\Gamma.
\]
We then have the following theorem.

\begin{thm}\label{thm:posetWincom}\cite[Theorem 2.11]{BasuRoy2022}
Let $\lang$ be a non-empty set, $(P,\le)$ be a poset with at least two incomparable elements, and $\mathcal{V},\models_\mathcal{V}$ be as above. If there exists a $\widehat{v}\in\mathcal{V}$ such that $\widehat{v}(\gamma)$ and $\widehat{v}(\delta)$ are incomparable for some $\gamma,\delta\in\lang$, then $\langle\lang,\models_\mathcal{V}\rangle$ is NF-paraconsistent.
\end{thm}

Our next example is taken, with some change of notation, from \cite{Carnielli2007}, where this is used in a different context.

\begin{exa}\cite[Example 2.13]{BasuRoy2022}
Let $\mathcal{S}=\langle\lang,\vdash\rangle$ be the logic, where $\lang=\mathbb{R}$, the set of real numbers, and $\vdash\,\subseteq\pow(\lang)\times\lang$ is defined as follows. For any $\Gamma\cup\{\alpha\}\subseteq\lang$,
\[
\Gamma\vdash\alpha\quad\hbox{iff}\quad\left\{\begin{array}{l}
     \alpha\in\Gamma,\,\hbox{or}\\
     \alpha=\frac{1}{n}\,\hbox{ for some natural number }\,n\in\mathbb{N},\,n\ge1,\,\hbox{or}\\
     \hbox{there is a sequence }\,(\alpha_n)_{n\in\mathbb{N}}\,\hbox{ in }\,\Gamma\,\hbox{ that converges to }\,\alpha.
\end{array}\right.
\]
It is straightforward to see that for any $\alpha,\beta,\gamma\in\lang$ such that $\gamma\neq\alpha,\beta$ and $\gamma\neq\frac{1}{n}$ for any $n\in\mathbb{N}$, $\{\alpha,\beta\}\not\vdash\gamma$. Thus, $\mathcal{S}$ is NF-paraconsistent.

The logic $\mathcal{S}$ is reflexive and monotone but not transitive. To see the latter, one may note that for any $n\in\mathbb{N}$, $\vdash\frac{1}{n}$ and $\left\{\frac{1}{n}\mid\,n\in\mathbb{N}\right\}\vdash0$ but $\not\vdash0$. This phenomenon is also noted in \cite{Carnielli2007}.
\end{exa}

The notion of \emph{$q$-consequence} (\emph{quasi-consequence}) was introduced in \cite{Malinowski1990} to produce a reasoning scheme in which ``rejection'' and ``acceptance'' of a proposition are not necessarily reciprocal to each other. Before proceeding to show that logics equipped with certain $q$-consequence operators are NF-paraconsistent, we describe the concept of a $q$-consequence operator below.

\begin{dfn}\label{dfn:q-cons}
Let $\lang$ be a set and $W:\pow(\lang)\to\pow(\lang)$. $W$ is called a \emph{$q$-consequence operator} on $\lang$ if the following conditions hold.
\begin{enumerate}[label=(\roman*)]
    \item For all $\Gamma\cup\Sigma\subseteq\lang$, if $\Gamma\subseteq\Sigma$, then $W(\Gamma)\subseteq W(\Sigma)$.
    \item For all $\Gamma\subseteq\lang$, $W(\Gamma\cup W(\Gamma))=W(\Gamma)$.
\end{enumerate}
Given such a $q$-consequence operator $W$ on $\lang$, the pair $\langle\lang,\vdash^W\rangle$ is called the logic induced by $W$, where $\vdash^W\,\subseteq\,\pow(\lang)\times\lang$ is defined as follows. For any $\Gamma\cup\{\alpha\}\subseteq\lang$, $\Gamma\vdash^W\alpha$ iff $\alpha\in W(\Gamma)$.
\end{dfn}

\begin{thm}\cite[Theorem 2.15]{BasuRoy2022}
Suppose $\mathcal{S}=\langle\lang,\vdash^W\rangle$ is the logic induced by a $q$-consequence operator $W$ on the set $\lang$. If $W(\lang)\subsetneq\lang$, then $\mathcal{S}$ is NF-paraconsistent.
\end{thm}

\subsection{NF-paraconsistency and logics of variable inclusion}

The \emph{logics of variable inclusion} have recently been rigorously studied, e.g., in \cite{Bonzio2020,BonzioPaoliPraBaldi2022,PaoliPraBaldiSzmuc2021}. The following definitions can be found there. For these logics, we assume that $\lang$ is the set of formulas defined inductively in the usual way over a denumerable set of variables $V$ using a finite set of connectives/operators called the \emph{signature/type}. In other words, $\lang$ is the formula algebra over $V$ of some type. The formula algebra has the universal mapping property for the class of all algebras of the same type as $\lang$ over $V$, i.e., any function $f:V\to A$, where $A$ is the universe of an algebra $\mathbf{A}$ of the same type as $\mathcal{L}$, can be uniquely extended to a homomorphism from $\lang$ to $\mathbf{A}$ (see \cite{FontJansanaPigozzi2003,Font2016} for more details). For the discussion in this subsection, we will assume that for any logic $\langle\lang,\vdash\rangle$, $\lang$ is as described above.

We first deal with the \emph{left variable inclusion companions}. These are described below.

\begin{dfn}\label{dfn:lvarinclusion}
Suppose $\mathcal{S}=\langle\lang,\vdash\rangle$ is a logic. For any $\alpha\in\lang$, let $\var(\alpha)$ denote the set of all the variables occurring in $\alpha$, and $\var(\Delta)=\displaystyle\bigcup_{\alpha\in\Delta}\var(\alpha)$.
\begin{enumerate}[label=(\roman*)]
    \item The \emph{left variable inclusion companion} of $\mathcal{S}$, denoted by $\mathcal{S}^{l}=\langle\lang,\vdash^{l}\rangle$, is defined as follows. For any $\Gamma\cup\{\alpha\}\subseteq\lang$, 
    \[
    \Gamma\vdash^{l}\alpha\;\hbox{iff there is a}\;\Delta\subseteq\Gamma\;\hbox{such that}\;\var(\Delta)\subseteq\var(\alpha)\;\hbox{and}\;\Delta\vdash\alpha.
    \]
    \item The \emph{pure left variable inclusion companion} of $\mathcal{S}$, denoted by $\mathcal{S}^{pl}=\langle\lang,\vdash^{pl}\rangle$, is defined as follows. For any $\Gamma\cup\{\alpha\}\subseteq\lang$,
    \[
    \Gamma\vdash^{pl}\alpha\;\hbox{iff there is a}\;\Delta\subseteq\Gamma\;\hbox{such that}\;\Delta\neq\emptyset,\var(\Delta)\subseteq\var(\alpha),\;\hbox{and}\;\Delta\vdash\alpha.
    \]
\end{enumerate}
\end{dfn}

\begin{thm}\label{thm:NFpara->Fpara(Sl,Spl)}
    Suppose $\mathcal{S}=\langle\lang,\vdash\rangle$ is a logic with $\bot\in\lang$, a falsity constant for $\mathcal{S}$ (i.e., $\mathcal{S}$ is not $\bot$-paraconsistent). Then $\mathcal{S}^l=\langle\lang,\vdash^l\rangle$ and $\mathcal{S}^{pl}=\langle\lang,\vdash^{pl}\rangle$ are not NF-paraconsistent.
\end{thm}

\begin{proof}
    Since $\bot$ is a falsity constant for $\mathcal{S}$, $\bot\vdash\alpha$ for all $\alpha\in\lang$.

    We note that $\{\bot\}$ is a non-empty subset of itself, and $\var(\bot)=\emptyset\subseteq\var(\alpha)$ for all $\alpha\in\lang$. Thus, $\bot\vdash^l\alpha$ and $\bot\vdash^{pl}\alpha$ for all $\alpha\in\lang$. Hence, $\bot$ is also a falsity constant for $\mathcal{S}^l$ and $\mathcal{S}^{pl}$.

    Now, suppose $\Gamma\vdash^l\alpha$ ($\Gamma\vdash^{pl}\alpha$). Then there exists $\Delta\subseteq\Gamma$ ($\emptyset\neq\Delta\subseteq\Gamma$) such that $\var(\Delta)\subseteq\var(\alpha)$ and $\Delta\vdash\alpha$. Then for any $\Gamma^\prime\supseteq\Gamma$, $\Delta\subseteq\Gamma^\prime$, and hence $\Gamma^\prime\vdash^l\alpha$ (respectively, $\Gamma^\prime\vdash^{pl}\alpha$). Thus, $\mathcal{S}^l$ and $\mathcal{S}^{pl}$ are both monotonic logics (regardless of whether $\mathcal{S}$ is monotonic or not).

    Hence, by Corollary \ref{cor:NFpara->Fpara}, $\mathcal{S}^l$ and $\mathcal{S}^{pl}$ are not NF-paraconsistent.
\end{proof}

\begin{rem}
    The above theorem shows that if the (pure) left variable inclusion companion of a logic $\mathcal{S}=\langle\lang,\vdash\rangle$ is NF-paraconsistent, then there cannot be a falsity constant for $\mathcal{S}$ in $\lang$, and $\mathcal{S},\mathcal{S}^l, \mathcal{S}^{pl}$ are all $\bot$-paraconsistent.
\end{rem}

\begin{exa}\label{exa:PWK-NotNFpara}
Paraconsistent weak Kleene logic (PWK) (described in \cite{Bonzio2017}) and intuitionistic paraconsistent weak Kleene logic (IPWK) (described in \cite{BasuChakraborty2021}) are the left variable inclusion companions of classical and intuitionistic propositional logics, respectively. We know that the classical and intuitionistic propositional logics (described in a language with a falsity constant) are not $\bot$-paraconsistent. Thus, by Theorem \ref{thm:NFpara->Fpara(Sl,Spl)}, PWK and IPWK are not NF-paraconsistent, although they are paraconsistent. 
\end{exa}

We first prove the following lemma before proceeding to a partial converse to Theorem \ref{thm:NFpara->Fpara(Sl,Spl)}.

\begin{lem}\label{lem:Spl-trivial}
    Suppose $\mathcal{S}=\langle\lang,\vdash\rangle$ is a logic. If $\alpha,\beta\in\lang$ such that $C_{\vdash^{pl}}(\{\alpha,\beta\})=\lang$, then either $\var(\alpha)=\emptyset$ or $\var(\beta)=\emptyset$ (i.e., either $\alpha$ or $\beta$ is a constant). Moreover, if $\not\vdash q$ for all $q\in V$, the same conclusion holds for $\mathcal{S}^l$, i.e., if $\alpha,\beta\in\lang$ such that $C_{\vdash^l}(\{\alpha,\beta\})=\lang$, then either $\var(\alpha)=\emptyset$ or $\var(\beta)=\emptyset$.
\end{lem}

\begin{proof}
    Suppose $\alpha,\beta\in\lang$ such that $C_{\vdash^{pl}}(\{\alpha,\beta\})=\lang$. Suppose further that $\var(\alpha)\neq\emptyset$. 
    
    Let $q\in V\setminus\var(\{\alpha,\beta\})$. Such a variable $q$ exists as $V$ is denumerable, while $\var(\{\alpha,\beta\})$ is finite. Now, $\{\alpha,\beta\}\vdash^{pl}q$. So, there exists $\Delta\subseteq\{\alpha,\beta\}$ such that $\Delta\neq\emptyset$, $\var(\Delta)\subseteq\var(q)=\{q\}$, and $\Delta\vdash q$. Clearly, $\var(\Delta)\subseteq\var(\{\alpha,\beta\})$. However, $q\notin\var(\{\alpha,\beta\})$. So, $\var(\Delta)=\emptyset$. Since $\var(\alpha)\neq\emptyset$, $\alpha\notin\Delta$. Then as $\Delta\neq\emptyset$, $\Delta=\{\beta\}$. So, $\var(\beta)=\var(\Delta)=\emptyset$.

    The argument for $\mathcal{S}^l$ proceeds in an identical manner with $\vdash^{pl}$ replaced by $\vdash^l$. The only other difference is that, in this case, the set $\Delta\subseteq\{\alpha,\beta\}$ can be empty. However, if $\Delta=\emptyset$, then $\vdash q$, which contradicts the additional assumption in this case that $\not\vdash q$ for all $q\in V$. Thus, it must be that $\var(\beta)=\emptyset$ if $\var(\alpha)\neq\emptyset$.
\end{proof}

\begin{thm}\label{thm:Sl-NFpara}
Suppose $\mathcal{S}=\langle\lang,\vdash\rangle$ is a logic. If $\var(\varphi)\neq\emptyset$ for all $\varphi\in\lang$, then $\mathcal{S}^{pl}$ is NF-paraconsistent.
Moreover, if $\not\vdash q$ for all $q\in V$, then $\mathcal{S}^{l}$ is also NF-paraconsistent.
\end{thm}

\begin{proof}
Suppose $\var(\varphi)\neq\emptyset$ for all $\varphi\in\lang$, but $\mathcal{S}^{pl}$ is not NF-paraconsistent. Let $p\in V$. So, by gECQ, there exists $\beta\in\lang$ such that $C_{\vdash^{pl}}(\{p,\beta\})=\lang$. Then by Lemma \ref{lem:Spl-trivial}, $\var(\beta)=\emptyset$ since $\var(p)=\{p\}$. This contradicts our assumption that $\var(\varphi)\neq\emptyset$ for all $\varphi\in\lang$. Thus, $\mathcal{S}^{pl}$ must be NF-paraconsistent.

In the case of $\mathcal{S}^l$, we use the additional hypothesis that $\not\vdash q$ for all $q\in V$ and arrive at a similar contradiction using Lemma \ref{lem:Spl-trivial}. Thus, under this additional hypothesis, $\mathcal{S}^l$ is also NF-paraconsistent.
\end{proof}

\begin{rem}
    We note that Theorem \ref{thm:Sl-NFpara} is not the exact converse of Theorem \ref{thm:NFpara->Fpara(Sl,Spl)}. While $\var(\varphi)\neq\emptyset$ for all $\varphi\in\lang$ (i.e., there are no logical constants) implies that there is no falsity constant, the former is clearly a stronger assumption. We can, however, achieve the full converse if the logic $\mathcal{S}$ is \emph{substitution invariant} (described below). The full converse is also obtained, in the case of $\mathcal{S}^{pl}$, if $\mathcal{S}$ is a transitive logic, and, in the case of $\mathcal{S}^l$, if $\mathcal{S}$ is monotonic and transitive.
\end{rem}    

Suppose $\lang$ is a formula algebra over a set of variables $V$. A \emph{substitution} is any function $\sigma:V\to\lang$ that extends to a unique endomorphism (also denoted by $\sigma$) from $\lang$ to itself via the universal mapping property. The logic $\mathcal{S}=\langle\lang,\vdash\rangle$ (and the relation $\vdash$) is said to be \emph{substitution invariant} or \emph{structural} if the following property holds. For any substitution $\sigma$ and for any $\Gamma\cup\{\varphi\}\subseteq\lang$, if $\Gamma\vdash\varphi$, then $\sigma(\Gamma)\vdash\sigma(\varphi)$. The property is referred to as \emph{structurality}. For more details and context, see \cite{FontJansanaPigozzi2003,Font2016}.

\begin{thm}\label{thm:Sl-NFpara2}
Suppose $\mathcal{S}=\langle\lang,\vdash\rangle$ is a structural logic. If there is no falsity constant for $\mathcal{S}$ in $\lang$, then $\mathcal{S}^{pl}$ is NF-paraconsistent.
Moreover, if there exists $\varphi\in\lang$ such that $\not\vdash\varphi$, then $\mathcal{S}^{l}$ is also NF-paraconsistent.
\end{thm}

\begin{proof}
    Suppose there is no falsity constant for $\mathcal{S}$ in $\lang$, but $\mathcal{S}^{pl}$ is not NF-paraconsistent. Let $p\in V$. So, by gECQ, there exists $\beta\in\lang$ such that $C_{\vdash^{pl}}(\{p,\beta\})=\lang$. Then, as in the proof of Lemma \ref{lem:Spl-trivial}, let $q\in V\setminus\var(\{p,\beta\})$. So, $\{p,\beta\}\vdash^{pl}q$, which implies that there exists $\Delta\subseteq\{p,\beta\}$ such that $\Delta\neq\emptyset$, $\var(\Delta)\subseteq\{q\}$, and $\Delta\vdash q$. So, by the same arguments as in the proof of Lemma \ref{lem:Spl-trivial}, $\Delta=\{\beta\}$ and $\var(\beta)=\var(\Delta)=\emptyset$. Thus, $\beta$ is a constant, and $\beta\vdash q$. So, by structurality, $\beta\vdash\alpha$ for all $\alpha\in\lang$, which implies that $\beta$ is a falsity constant for $\mathcal{S}$. This is a contradiction. Hence $\mathcal{S}^{pl}$ is NF-paraconsistent.

    The argument for $\mathcal{S}^l$ proceeds similarly with $\vdash^{pl}$ replaced by $\vdash^l$. The extra case we need to deal with now is where $\Delta=\emptyset$. However, that implies $\vdash q$, and hence by structurality, $\vdash\alpha$ for all $\alpha\in\lang$. This contradicts the additional assumption that there exists $\varphi\in\lang$ such that $\not\vdash\varphi$. Thus, $\mathcal{S}^l$ is also NF-paraconsistent.
\end{proof}

\begin{thm}\label{thm:Spl-NFpara}
Suppose $\mathcal{S}=\langle\lang,\vdash\rangle$ is a transitive logic. If there is no falsity constant for $\mathcal{S}$ in $\lang$, then $\mathcal{S}^{pl}$ is NF-paraconsistent.
\end{thm}

\begin{proof}
    Suppose $\mathcal{S}^{pl}$ is not NF-paraconsistent. Let $p\in V$. So, by gECQ, there exists $\beta\in\lang$ such that $C_{\vdash^{pl}}(\{p,\beta\})=\lang$. Then, by Lemma \ref{lem:Spl-trivial}, $\var(\beta)=\emptyset$ since $\var(p)\neq\emptyset$. Thus, $\beta$ is a constant.
    
    Let $\gamma\in\lang$ such that $p\notin\var(\gamma)$. Now, $\{p,\beta\}\vdash^{pl}\gamma$. So, there exists $\Delta\subseteq\{p,\beta\}$ such that $\Delta\neq\emptyset$, $\var(\Delta)\subseteq\var(\gamma)$, and $\Delta\vdash\gamma$. Since $p\notin\var(\gamma)$, $p\notin\Delta$. Thus, $\Delta=\{\beta\}$, and $\beta\vdash\gamma$. Hence we can conclude that $\beta$ is a constant such that $\beta\vdash\gamma$ for all $\gamma\in\lang$ with $p\notin\var(\gamma)$.

    Suppose $\beta$ is not a falsity constant. Then there exists $\varphi\in\lang$ such that $\beta\not\vdash\varphi$. So, by the above arguments, $p\in\var(\varphi)$. Now, let $p^\prime\in V\setminus\var(\varphi)$. Again by gECQ, there exists $\psi\in\lang$ such that $C_{\vdash^{pl}}(\{p^\prime,\psi\})=\lang$. Then, again by Lemma \ref{lem:Spl-trivial}, $\var(\psi)=\emptyset$ since $\var(p^\prime)\neq\emptyset$. This, in particular, implies that $p\notin\var(\psi)$, and thus, $\beta\vdash\psi$. 

    Now, since $C_{\vdash^{pl}}(\{p^\prime,\psi\})=\lang$, $\{p^\prime,\psi\}\vdash^{pl}\varphi$. So, there exists $\Delta\subseteq\{p^\prime,\psi\}$ such that $\Delta\neq\emptyset$, $\var(\Delta)\subseteq\var(\varphi)$, and $\Delta\vdash\varphi$. Since $p^\prime\notin\var(\varphi)$, $p^\prime\notin\Delta$. Thus, $\Delta=\{\psi\}$ and so, $\psi\vdash\varphi$. Since $\beta\vdash\psi$ and $\psi\vdash\varphi$, by transitivity, we have $\beta\vdash\varphi$. This is a contradiction to our earlier assumption. Hence $\mathcal{S}^{pl}$ must be NF-paraconsistent.
\end{proof}

\begin{thm}\label{thm:Sl-NFpara3}
Suppose $\mathcal{S}=\langle\lang,\vdash\rangle$ is a monotonic and transitive logic such that $\not\vdash q$ for all $q\in V$. If there is no falsity constant for $\mathcal{S}$ in $\lang$, then $\mathcal{S}^{l}$ is NF-paraconsistent.
\end{thm}

\begin{proof}
    Suppose $\mathcal{S}^l$ is not NF-paraconsistent. Let $p\in V$. So, by gECQ, there exists $\beta\in\lang$ such that $C_{\vdash^l}(\{p,\beta\})=\lang$. Then, by Lemma \ref{lem:Spl-trivial}, $\var(\beta)=\emptyset$ since $\var(p)\neq\emptyset$. Thus, $\beta$ is a constant.

    Let $\gamma\in\lang$. If $\vdash\gamma$, then $\beta\vdash\gamma$ by monotonicity. Now, suppose $\not\vdash\gamma$ and $p\notin\var(\gamma)$. Since $V$ is denumerable and $\not\vdash q$ for all $q\in V$, such a $\gamma$ exists. Then $\{p,\beta\}\vdash^l\gamma$. This implies that there exists $\Delta\subseteq\{p,\beta\}$ such that $\var(\Delta)\subseteq\var(\gamma)$, and $\Delta\vdash\gamma$. Since $p\notin\var(\gamma)$, $p\notin\Delta$. Then as $\not\vdash\gamma$, $\Delta=\{\beta\}$, and $\beta\vdash\gamma$. Thus, for any $\gamma\in\lang$, if $\vdash\gamma$ or $p\notin\var(\gamma)$, then $\beta\vdash\gamma$. In other words, $\beta$ is a constant such that for all $\gamma\in\lang$, if $\beta\not\vdash\gamma$, then $\not\vdash\gamma$ and $p\in\var(\gamma)$. In particular, this implies that $\beta\vdash\sigma$ for all $\sigma\in\lang$ such that $\var(\sigma)=\emptyset$.
    
    Suppose $\beta$ is not a falsity constant. Then there exists $\varphi\in\lang$ such that $\beta\not\vdash\varphi$. So, by the above arguments, $\not\vdash\varphi$ and $p\in\var(\varphi)$. Now, let $p^\prime\in V\setminus\var(\varphi)$ (such a variable exists since $V$ is denumerable). Again by gECQ, there exists $\psi\in\lang$ such that $C_{\vdash^l}(\{p^\prime,\psi\})=\lang$. Then, again by Lemma \ref{lem:Spl-trivial}, $\var(\psi)=\emptyset$ since $\var(p^\prime)\neq\emptyset$. Thus, $\beta\vdash\psi$.

    Now, since $C_{\vdash^l}(\{p^\prime,\psi\})=\lang$, $\{p^\prime,\psi\}\vdash^l\varphi$. So, there exists $\Delta\subseteq\{p^\prime,\psi\}$ such that $\var(\Delta)\subseteq\var(\varphi)$, and $\Delta\vdash\varphi$. Since $p^\prime\notin\var(\varphi)$, $p^\prime\notin\Delta$. Moreover, $\not\vdash\varphi$, which implies that $\Delta\neq\emptyset$. Thus, $\Delta=\{\psi\}$ and so, $\psi\vdash\varphi$. Since $\beta\vdash\psi$ and $\psi\vdash\varphi$, by transitivity, we have $\beta\vdash\varphi$. This is a contradiction to our earlier assumption. Hence $\mathcal{S}^l$ must be NF-paraconsistent.
\end{proof}

\begin{rem}
    We note that we do not require the full strength of the monotonicity condition in the above theorem. In fact, we only need the following weaker condition. For any $\gamma\in\lang$, if $\vdash\gamma$, then $\beta\vdash\gamma$ for all $\beta\in\lang$.

    One may also note that Theorems \ref{thm:Spl-NFpara} and \ref{thm:Sl-NFpara3} do not require the full strength of the transitivity condition. The following weaker ``pointwise transitivity'' is enough. For any $\alpha,\beta,\gamma\in\lang$, if $\alpha\vdash\beta$ and $\beta\vdash\gamma$, then $\alpha\vdash\gamma$.
\end{rem}

Next, we explore the \emph{right variable inclusion companions}. These are described below.

\begin{dfn}\label{dfn:rvarinclusion}
Suppose $\mathcal{S}=\langle\lang,\vdash\rangle$ is a logic.
    \begin{enumerate}[label=(\roman*)]
    \item The \emph{right variable inclusion companion} of $\mathcal{S}$, denoted by $\mathcal{S}^r=\langle\lang\vdash^r\rangle$, is defined as follows. For any $\Gamma\cup\{\alpha\}\subseteq\lang$, 
    \[
    \Gamma\vdash^{r}\alpha\;\hbox{iff either}\;\Gamma\;\hbox{contains an \emph{$\mathcal{S}$-antitheorem}, or}\;\Gamma\vdash\alpha\;\hbox{and}\;\var(\alpha)\subseteq\var(\Gamma).
    \]
    An $\mathcal{S}$-antitheorem is a set $\Sigma\subseteq\lang$ such that for every substitution $\sigma$, $\sigma(\Sigma)\vdash\varphi$ for all $\varphi\in\lang$, i.e., $\sigma(\Sigma)$ \emph{explodes} in $\mathcal{S}$ for every $\sigma$. 
    \item The \emph{pure right variable inclusion companion} of $\mathcal{S}$, denoted by $\mathcal{S}^{pr}=\langle\lang\vdash^{pr}\rangle$, is defined as follows. For any $\Gamma\cup\{\alpha\}\subseteq\lang$, 
    \[    \Gamma\vdash^{pr}\alpha\;\hbox{iff}\;\Gamma\vdash\alpha\;\hbox{and}\;\var(\alpha)\subseteq\var(\Gamma).
    \]
    \end{enumerate}
\end{dfn}

\begin{rem}
    The above definition of a right variable inclusion companion agrees with the ones in \cite{PaoliPraBaldiSzmuc2021,PraBaldi2020}. However, in \cite{BonzioPaoliPraBaldi2022}, these are defined with the following slight variation.
    For any $\Gamma\cup\{\alpha\}\subseteq\lang$, 
    \[
    \Gamma\vdash^{r}\alpha\;\hbox{iff either}\;\Gamma\;\hbox{is an $\mathcal{S}$-antitheorem, or}\;\Gamma\vdash\alpha\;\hbox{and}\;\var(\alpha)\subseteq\var(\Gamma).
    \]
    The same has been indicated in \cite{CiuniFergusonSzmuc2019}. This difference is, of course, inconsequential in the presence of monotonicity. We choose to work with the weaker definition here as we aim to be frugal with our assumptions.
\end{rem}

It is easy to see that the pure right variable inclusion companion of a logic $\mathcal{S}$ is paraconsistent with respect to a unary (negation) operator $\neg$, provided there are at least two distinct variables. The following result tells us that they are also NF-paraconsistent provided there is an ample supply of variables.

\begin{thm}\label{thm:Spr-NFpara}
    Suppose $\mathcal{S}=\langle\lang,\vdash\rangle$ is a logic. Then $\mathcal{S}^{pr}=\langle\lang,\vdash^{pr}\rangle$ is NF-paraconsistent. 
\end{thm}

\begin{proof}
    Suppose $\mathcal{S}^{pr}$ is not NF-paraconsistent. Let $p\in V$. So, by gECQ, there exists $\beta\in\lang$ such that $C_{\vdash^{pr}}(\{p,\beta\})=\lang$.

    Let $q\in V\setminus\var(\{p,\beta\})$. So, $\{p,\beta\}\vdash^{pr}q$. This implies that $q\in\var(\{p,\beta\})$, which contradicts our assumption about $q$. Hence $\mathcal{S}^{pr}$ must be NF-paraconsistent.
\end{proof}

The right variable inclusion companion of a logic is not paraconsistent, in general. For example, Bochvar logic, $B_3$, which is the right variable inclusion companion of classical propositional logic, is not paraconsistent (see \cite{BonzioPaoliPraBaldi2022,PraBaldi2020}). Hence $B_3$ is not NF-paraconsistent, by Theorem \ref{thm:NFpara->para}.

We can, however, show that if a logic $\mathcal{S}=\langle\lang,\vdash\rangle$ is NF-paraconsistent, then its right variable inclusion companion is also NF-paraconsistent, provided $\mathcal{S}$ satisfies a weak form of monotonicity, viz., monotonicity for the trivial subsets of $\lang$. This can be described as follows. For any $\Gamma\subseteq\lang$, if $\Gamma$ is $\mathcal{S}$-trivial, i.e., if $C_\vdash(\Gamma)=\lang$, then every superset of $\Gamma$ is also $\mathcal{S}$-trivial, i.e., $C_\vdash(\Gamma^\prime)=\lang$ for all $\Gamma^\prime\supseteq\Gamma$.

\begin{lem}\label{lem:Sr-trivial:S-trivial}
    Suppose $\mathcal{S}=\langle\lang,\vdash\rangle$ is a logic. Then every finite $\mathcal{S}^r$-trivial set contains an $\mathcal{S}$-trivial set.
\end{lem}

\begin{proof}
    Suppose $\Gamma$ is a finite $\mathcal{S}^r$-trivial set. Let $q\in V\setminus\var(\Gamma)$. Then $\Gamma\vdash^r q$. So, either $q\in\var(\Gamma)$ and $\Gamma\vdash q$, or $\Gamma$ contains an $\mathcal{S}$-antitheorem. Since $q\notin\var(\Gamma)$, $\Gamma$ must contain an $\mathcal{S}$-antitheorem. Let $\Delta\subseteq\Gamma$ be an $\mathcal{S}$-antitheorem. This implies that $C_\vdash(\sigma(\Delta))=\lang$ for every substitution $\sigma$. Thus, in particular, $C_\vdash(\Delta)=\lang$ ($\sigma$ is the identity endomorphism from $\lang$ to itself in this case). Hence $\Delta$ is an $\mathcal{S}$-trivial set contained in $\Gamma$.
\end{proof}

\begin{thm}\label{thm:S-NFpara->Sr-NFpara}
    Suppose $\mathcal{S}=\langle\lang,\vdash\rangle$ that satisfies the property of monotonicity for trivial sets. If $\mathcal{S}$ is NF-paraconsistent, then $\mathcal{S}^r$ is also NF-paraconsistent.
\end{thm}

\begin{proof}
    Suppose $\mathcal{S}^r$ is not NF-paraconsistent. Let $\alpha\in\lang$. Then by gECQ, there exists $\beta\in\lang$, such that $C_{\vdash^r}(\{\alpha,\beta\})=\lang$. So, $\{\alpha,\beta\}$ is an $\mathcal{S}^r$-trivial set. Hence, by Lemma \ref{lem:Sr-trivial:S-trivial}, there must be an $\mathcal{S}$-trivial $\Delta\subseteq\{\alpha,\beta\}$. Then, by monotonicity for trivial sets, we have that $\{\alpha,\beta\}$ is also $\mathcal{S}$-trivial, i.e., $C_\vdash(\{\alpha,\beta\})=\lang$. This implies that for every $\alpha\in\lang$, there exists $\beta\in\lang$ such that $C_\vdash(\{\alpha,\beta\})=\lang$. Thus, $\mathcal{S}$ satisfies gECQ and, therefore, is not NF-paraconsistent. Hence if $\mathcal{S}$ satisfies monotonicity for trivial sets and is NF-paraconsistent, then $\mathcal{S}^r$ is also NF-paraconsistent.
\end{proof}

\subsection{A quasi-negation}
Suppose $\mathcal{S}=\langle\lang,\vdash\rangle$ is a logic. For any $\alpha,\beta\in\lang$, if $C_\vdash(\{\alpha,\beta\})=\lang$, then $\beta$ acts like a negation of $\alpha$. This can be justified by noting that in case $\mathcal{S}$ is not paraconsistent with respect to a unary operator $\neg$ in $\mathcal{S}$, i.e., $\neg$-ECQ holds, then $\neg\alpha$ is indeed such a $\beta$. However, in general, such a $\beta$ need not be unique. For example, in the case of classical propositional logic (CPC) or intuitionistic propositional logic (IPC), we have $C_\vdash(\{\alpha,\alpha\limp\neg\alpha\})=\lang$. Thus, $\alpha\limp\neg\alpha$ is also an instance of such a $\beta$. Motivated by this, we now have the following definition.

\begin{dfn}\label{dfn:quasi-neg}
Suppose $\mathcal{S}=\langle\lang,\vdash\rangle$ is a logic and $\alpha\in\lang$. Then a \emph{quasi-negation} of $\alpha$ is any $\beta\in\lang$ such that $C_\vdash(\{\alpha,\beta\})=\lang$. The set of all quasi-negations of $\alpha$ is denoted by $QN(\alpha)$. Lastly, we denote a quasi-negation of $\alpha$ by $\qneg\alpha$, provided $QN(\alpha)\neq\emptyset$. 
\end{dfn}

\begin{rem}
    The symbol $\qneg$ is not a connective or logical operator. Given $\alpha\in\lang$, if $QN(\alpha)\neq\emptyset$, $\qneg\alpha$ represents any member of $QN(\alpha)$. Thus, in the case of a logic $\langle\lang,\vdash\rangle$, where $\lang$ is the set of formulas in the usual sense, the string $\qneg\alpha$ does not denote a single formula, but instead, it represents any formula in $QN(\alpha)\subseteq\lang$. Finally, this notation is not absolutely necessary since all the results involving quasi-negation can be stated in terms of $QN(\cdot)$. However, the notation offers a certain convenience and perhaps helps to visualize it as a negation.
\end{rem}

We now explore some properties of the quasi-negation in the following theorem.

\begin{thm}\cite[Theorem 2.17]{BasuRoy2022}
Suppose $\mathcal{S}=\langle\lang,\vdash\rangle$ is a logic and $\alpha,\beta\in\lang$. 
\begin{enumerate}[label=(\arabic*)]
    \item $\alpha$ is a $\qneg\qneg\alpha$, i.e., $\alpha\in QN(\qneg\alpha)$. (Quasi double negation)
    
    \item $\alpha$ is a $\qneg\beta$ iff $\beta$ is a $\qneg\alpha$, i.e., $\alpha\in QN(\beta)$ iff $\beta\in QN(\alpha)$. (Quasi contraposition).
    \item If $\mathcal{S}$ is Tarskian and $\alpha\vdash\beta$, then any $\qneg\beta$ is a $\qneg\alpha$, i.e., $QN(\beta)\subseteq QN(\alpha)$. 
\end{enumerate}
\end{thm}

The following result establishes some connections between the above defined quasi-negation and the concepts of paraconsistency and NF-paraconsistency.

\begin{thm}\cite[Theorem 2.18]{BasuRoy2022}
Suppose $\mathcal{S}=\langle\lang,\vdash\rangle$ is a logic. 
\begin{enumerate}[label=(\arabic*)]
    \item Suppose $\mathcal{S}$ has a unary operator $\neg$. Then $\mathcal{S}$ is paraconsistent with respect to $\neg$ iff there exists an $\alpha\in\lang$ such that $\neg\alpha\notin QN(\alpha)$.
    \item $\mathcal{S}$ is NF-paraconsistent iff there exists an $\alpha\in\lang$ such that $QN(\alpha)=\emptyset$.
    \item If $\mathcal{S}$ is monotonic and there does not exist any nullary operator or constant $\bot\in\lang$ such that $QN(\bot)=\lang$, then $\mathcal{S}$ is $\bot$-paraconsistent.
\end{enumerate}
\end{thm}

\begin{rem}
It follows from the above theorem that a logic is NF-paraconsistent iff it is paraconsistent with respect to the quasi-negation.
\end{rem}

We end this subsection with a few properties of quasi-negation in the case of CPC $=\langle\lang,\vdash_{\mathrm{CPC}}\rangle$, where $\vdash_{\mathrm{CPC}}$ is the usual classical consequence relation, using the kite of negations in \cite{Dunn1999}. The following observations follow easily using the standard 2-element Boolean algebra interpretation of CPC $\langle\{0,1\},\land,\lor,\neg\rangle$. For any $\alpha\in\lang$, $QN(\alpha)\neq\emptyset$, i.e., there exists a $\qneg\alpha$ and $C_{\vdash_{\mathrm{CPC}}}(\{\alpha,\qneg\alpha\})=\lang$. This implies that for any valuation $v:\lang\to\{0,1\}$, if $v(\alpha)=1$, then $v(\qneg\alpha)=0$. Here, by $v(\qneg\alpha)=0$, we mean $v(\beta)=0$ for every $\beta\in QN(\alpha)$. Conversely, suppose $v:\lang\to\{0,1\}$ is a valuation such that $v(\qneg\alpha)=0$. Then, in particular, $v(\neg\alpha)=0$ since $\neg\alpha\in QN(\alpha)$. So, $v(\alpha)=1$. Thus, we can say that for any $\alpha\in\lang$  and any valuation $v:\lang\to\{0,1\}$, $v(\alpha)=1$ iff $v(\qneg\alpha)=0$. Now, for any $\alpha,\beta\in\lang$, the following properties are easy to check.

\begin{enumerate}[label=(\arabic*)]
    \item If $\alpha\vdash\beta$, then $\qneg\beta\vdash\qneg\alpha$ (Contraposition)
    \item $\alpha\vdash\neg\qneg\alpha$ and $\alpha\vdash\qneg\neg\alpha$ (Galois double negation)
    \item $\alpha\vdash\qneg\,\qneg\alpha$ (Constructive double negation)
    \item $\qneg\,\qneg\alpha\vdash\alpha$ (Classical double negation)
    \item $\alpha\land\qneg\alpha\vdash\beta$ (Absurdity)
\end{enumerate}

It is also easy to check that in the case of IPC, all the above properties, except (4), are satisfied by the quasi-negation. Thus, at least for CPC and IPC, the quasi-negations of a formula behave in much the same way as the regular classical and intuitionistic negation operators, respectively.

\section{Principles of explosion}
In this section, we discuss further generalizations of ECQ. One can develop a notion of paraconsistency via the failure of each of these. We have, however, kept the discussion here restricted to principles of explosion only. Suppose $\mathcal{S}=\langle\lang,\vdash\rangle$ is a logic. As mentioned in the introductory section, a principle of explosion in $\mathcal{S}$ might be described, in general, as one that allows a set $\Gamma\subsetneq\lang$ to blow up to the whole set $\lang$, i.e., to result in $C_\vdash(\Gamma)=\lang$.

The first principle of explosion is called sECQ (the `s' here denotes that this is a `set-based' principle of explosion). This can be described as follows. 
\begin{center}
   For each $\alpha\in\lang$, there exists a set $\Gamma\subsetneq\lang$ such that $\alpha\in\Gamma$ and $C_\vdash(\Gamma)=\lang$. 
\end{center}

\begin{thm}\label{thm:gECQ->sECQ}
    Suppose $\mathcal{S}=\langle\lang,\vdash\rangle$ is a logic.
    \begin{enumerate}[label=(\roman*)]
        \item If $\neg$-ECQ, where $\neg$ is a unary operator in $\mathcal{S}$, holds in $\mathcal{S}$, then sECQ also holds in $\mathcal{S}$, provided $\{\alpha,\neg\alpha\}\subsetneq\lang$ for all $\alpha\in\lang$.
        \item If gECQ holds in $\mathcal{S}$, then sECQ also holds in $\mathcal{S}$, provided $\lang$ has at least three distinct elements. 
    \end{enumerate}
\end{thm}

\begin{proof}
    The proof is straightforward from the formulations of $\neg$-ECQ, gECQ, and sECQ.
\end{proof}

\begin{rem}
   The condition that $\lang$ has at least three distinct elements is enough to 
   guarantee that if $\neg$-ECQ holds in $\mathcal{S}$, then sECQ also holds in it. Thus, the above theorem could have been stated more simply. However, the proviso that $\{\alpha,\neg\alpha\}\subsetneq\lang$ for all $\alpha$, in part (i) of the above theorem, is more general as this doesn't require that $\neg\alpha\neq\alpha$ for all $\alpha$.
   
   sECQ is clearly a generalization of ECQ and gECQ. Any paraconsistent logic that is not NF-paraconsistent, such as PWK, and IPWK (considered in Example \ref{exa:PWK-NotNFpara}), are instances of logics where gECQ holds, and hence by the above theorem, sECQ also holds. However, since these are paraconsistent, ECQ does not hold in these logics.
   
   The following are some further examples of logics, where ECQ fails but gECQ and sECQ hold. See Examples \ref{exa:sECQ!->spECQ} and \ref{exa:pfECQ1!->spECQ} for logics, where sECQ holds but gECQ fails. 
\end{rem}

\begin{exa}\label{exa:sECQ-P1}
    Let $\mathcal{S}=\langle\lang,\vdash\rangle$ be the logic $P^1$ proposed by Sette in \cite{Sette1973}. Here $\lang$ is the formula algebra over a denumerable set of propositional variables $V$, of type $\{\neg,\limp\}$, and $\vdash$ is a Tarskian consequence relation. Apart from the original Hilbert-style system presented in \cite{Sette1973}, alternative axiomatizations for $P^1$ can be found in \cite{Ciucura2015,Ciucura2020}. These articles also include semantics for $P^1$ and the soundness and completeness results. $P^1$ is a paraconsistent logic in the sense that ECQ fails in it but only for $p\in V$. In other words, for all $\alpha\in\lang\setminus V$, $C_\vdash(\{\alpha,\neg\alpha\})=\lang$. 

    It can be easily shown that $\vdash\alpha\limp\alpha$. Then, using the semantics and the completeness result provided in the above articles, it is straightforward to derive that $C_\vdash(\neg(\alpha\limp\alpha))=\lang$. Now, since $P^1$ is monotonic, $C_\vdash(\{\alpha,\neg(\alpha\limp\alpha)\})=\lang$ for all $\alpha\in\lang$. Thus, gECQ holds in $P^1$. (Instead of $\neg(\alpha\limp\alpha)$, we can use the negation of any theorem of $P^1$ to arrive at the same conclusion.) Since $\lang$ is denumerable and hence, has at least three distinct elements, by Theorem \ref{thm:gECQ->sECQ}, sECQ also holds in $P^1$.

    It is, however, easier to show that sECQ holds in $P^1$ directly as follows. Let $\alpha\in\lang$ and $\Gamma=\{\alpha,\neg\alpha,\neg\neg\alpha\}$. Then $\alpha\in\Gamma\subsetneq\lang$. Since $C_\vdash(\{\neg\alpha,\neg\neg\alpha\})=\lang$, by monotonicity, $C_\vdash(\Gamma)=\lang$. Hence, sECQ holds in $P^1$.
\end{exa}

\begin{rem}\label{rem:GentlyParaconsistent-sECQ}
    Sette's logic, $P^1$, is a member of the class of \emph{gently paraconsistent calculi/logics} discussed in \cite{Ciucura2020-1}. In each of these logics, ECQ fails. However, if $\langle\lang,\vdash\rangle$ is a gently paraconsistent logic, then for any $\alpha\in\lang$, $C_\vdash(\{\alpha,\neg\alpha,\neg\neg\alpha\})=\lang$, which is called the \emph{principle of gentle explosion}. Thus, clearly, sECQ holds.

    Each of the gently paraconsistent logics contains the positive fragment of classical propositional logic, and hence, as in $P^1$, $\alpha\limp\alpha$ is a theorem in each of them. Thus, by the same arguments as in the previous example, gECQ also holds in the gently paraconsistent logics.
\end{rem}

\begin{rem}\label{rem:Pm-sECQ}
    $P^1$ is also a member of the hierarchies of paraconsistent systems, $P^m$ and $P^{*m}$, where $m\ge1$, presented in \cite{Ciucura2020}. These logics use the same set of formulas $\lang$ as $P^1$, and ECQ fails in each of these systems. 
    
    However, in the logic $P^m=\langle\lang,\vdash\rangle$, where $m\ge1$, for any $\alpha\in\lang$, $C_\vdash(\{\alpha,\neg\alpha,\ldots,\neg^m\alpha,\neg^{m+1}\alpha\})=\lang$, where for any $n\ge1$, $\neg^n\alpha$ denotes $\underbrace{\neg\ldots\neg}_{n\hbox{ times}}\alpha$. Thus, sECQ holds in $P^m$ for each $m\ge1$. As in $P^1$, $\alpha\limp\alpha$ is a theorem of each $P^m$, and hence, by the same arguments as for $P^1$, gECQ also holds in these logics.

    The same arguments hold for the logics $P^{*m}$ as these logics have the same properties that are used for proving that gECQ and sECQ hold in $P^m$.
\end{rem}

The next principle of explosion is a slight variation of sECQ. We call this sECQ$^\prime$. This can be described as follows.
\begin{center}
   For each $\alpha\in\lang$, there exists a set $\Gamma\subsetneq\lang$ such that $\Gamma\cup\{\alpha\}\subsetneq\lang$ and $C_\vdash(\Gamma\cup\{\alpha\})=\lang$. 
\end{center}

\begin{thm}
    Suppose $\mathcal{S}=\langle\lang,\vdash\rangle$ is a logic. sECQ holds in $\mathcal{S}$ iff sECQ$^\prime$ holds in $\mathcal{S}$, i.e., sECQ is equivalent to sECQ$^\prime$.
\end{thm}

\begin{proof}
    Suppose sECQ holds in $\mathcal{S}$. Let $\alpha\in\lang$. Then there exists $\Gamma\subsetneq\lang$ such that $\alpha\in\Gamma$ and $C_\vdash(\Gamma)=\lang$. Now, let $\Gamma^\prime=\Gamma\setminus\{\alpha\}$. Then $\Gamma^\prime\subsetneq\lang$ such that $\Gamma^\prime\cup\{\alpha\}=\Gamma\subsetneq\lang$ and $C_\vdash(\Gamma^\prime\cup\{\alpha\})=C_\vdash(\Gamma)=\lang$. Thus, sECQ$^\prime$ holds in $\mathcal{S}$.

    Conversely, suppose sECQ$^\prime$ holds in $\mathcal{S}$. Let $\alpha\in\lang$. Then there exists $\Gamma\subsetneq\lang$ such that $\Gamma\cup\{\alpha\}\subsetneq\lang$ and $C_\vdash(\Gamma\cup\{\alpha\})=\lang$. Now, let $\Gamma^\prime=\Gamma\cup\{\alpha\}$. Then $\Gamma^\prime\subsetneq\lang$ such that $\alpha\in\Gamma^\prime$ and $C_\vdash(\Gamma^\prime)=\lang$. Thus, sECQ holds in $\mathcal{S}$.
\end{proof}

\begin{rem}
    In light of the above theorem, sECQ will henceforth denote both of the above principles of explosion: sECQ and sECQ$^\prime$.
\end{rem}

We note that all the principles of explosion that we have discussed so far, viz., ECQ, $\bot$-ECQ, gECQ, and even sECQ, are, in a sense, `pointwise.' Each of these principles revolves around elements of $\lang$, the underlying set in the logic $\langle\lang,\vdash\rangle$. In case of ECQ, gECQ, sECQ, for each $\alpha\in\lang$, there is some set containing $\alpha$ that explodes, and for $\bot$-ECQ, $\bot\in\lang$. We can indeed describe a principle of explosion using a set $\Gamma\subseteq\lang$ in the following way.
\begin{center}
    For all $\Gamma\subseteq\lang$ and for all $\alpha\in\lang$, if $\Gamma\vdash\alpha$ and $\Gamma\vdash\neg\alpha$, then $C_\vdash(\Gamma)=\lang$,
\end{center}
where $\neg$ is a unary operator on $\lang$. However, this is equivalent to $\neg$-ECQ if the logic is reflexive and transitive.
We now formulate some principles of explosion that differ from the above-discussed `pointwise' explosions and are not equivalent to ECQ, either. The first among these is obtained by turning sECQ$^\prime$ around as follows.
\begin{center}
    For all $\Gamma\subsetneq\lang$, there exists $\alpha\in\lang$ such that $\Gamma\cup\{\alpha\}\subsetneq\lang$ and $C_\vdash(\Gamma\cup\{\alpha\})=\lang$.
\end{center}
We call this spECQ (the letter-pair `sp' here denotes set-point in reference to the manner in which the explosion takes place). 

\begin{thm}\label{thm:spECQ->sECQ}
    Suppose $\mathcal{S}=\langle\lang,\vdash\rangle$ is a logic such that $\lang$ has at least two distinct elements. If spECQ holds in $\mathcal{S}$, then gECQ holds in it. Moreover, if $\lang$ has at least three distinct elements, then sECQ also holds in $\mathcal{S}$.
\end{thm}

\begin{proof}
    Suppose spECQ holds in $\mathcal{S}$. Let $\alpha\in\lang$. Since $\lang$ has at least two distinct elements, $\{\alpha\}\subsetneq\lang$. Then by spECQ, there exists $\beta\in\lang$ such that $\{\alpha\}\cup\{\beta\}\subsetneq\lang$ and $C_\vdash(\{\alpha\}\cup\{\beta\})=C_\vdash(\{\alpha,\beta\})=\lang$. Thus, gECQ holds in $\mathcal{S}$. Now, suppose $\lang$ has at least three distinct elements. Then, by Theorem \ref{thm:gECQ->sECQ}, sECQ also holds in $\mathcal{S}$.
\end{proof}

The following example shows that sECQ does not imply spECQ and gECQ. 

\begin{exa}\label{exa:sECQ!->spECQ}
    Let $\mathcal{S}=\langle\lang,\vdash\rangle$ be the logic, where $\lang=\NN\setminus\{0\}$, and $\vdash\,\subseteq\pow(\lang)\times\lang$ is such that, for any $\Gamma\subseteq\lang$,
    \[
    C_\vdash(\Gamma)=\left\{\begin{array}{ll}
         \lang&\hbox{if }\Gamma=\{n,n+1,\ldots,2n\}\hbox{ for some }n\in\lang,\\
         \Gamma&\hbox{otherwise}. 
    \end{array}\right.
    \]
    Let $k\in\lang$. Then, $\Gamma=\{k,k+1,\ldots,2k\}\subsetneq\lang$ is such that $k\in\Gamma$ and $C_\vdash(\Gamma)=\lang$. Thus, sECQ holds in $\mathcal{S}$.

    Now, to see that spECQ does not hold in $\mathcal{S}$, let $\Gamma=\{1,4\}$. We note that $\Delta=\{1,2\}$ is the only subset of $\lang$ such that $1\in\Delta$ and $C_\vdash(\Delta)=\lang$. Thus, $C_\vdash(\Gamma\cup\{\alpha\})\neq\lang$ for all $\alpha\in\lang$. Hence, spECQ does not hold in $\mathcal{S}$.

    Lastly, to see that gECQ does not hold in $\mathcal{S}$, we note that the only two element set that explodes is $\Delta=\{1,2\}$. Hence, for instance, there does not exist any $\beta\in\lang$ such that $C_\vdash\{3,\beta\}=\lang$. Hence, gECQ cannot hold in $\mathcal{S}$.
\end{exa}

The following are some other principles of explosion that are entirely free from any reference to elements of $\lang$.
\begin{itemize}
    \item For all $\Gamma\subsetneq\lang$, there exists $\Delta\subsetneq\lang$ such that $\Gamma\subseteq\Delta$ and $C_\vdash(\Delta)=\lang$. (pfECQ1)
    \item For all $\Gamma\subsetneq\lang$, there exists $\Delta\subsetneq\lang$ such that $\Gamma\cup\Delta\subsetneq\lang$ and $C_\vdash(\Gamma\cup\Delta)=\lang$. (pfECQ2)
    \item For all $\Gamma\subsetneq\lang$, there exists $\emptyset\neq\Delta\subsetneq\lang$ such that $\Gamma\cup\Delta\subsetneq\lang$, and for every $\emptyset\neq\Delta^\prime\subseteq\Delta$, $C_\vdash(\Gamma\cup\Delta^\prime)=\lang$. (pfECQ3)
\end{itemize}
The letter-pair `pf' in the above principles denotes point-free. We now investigate the relationships between these various principles of explosion.

\begin{thm}\label{thm:pfECQ1<->pfECQ2}
    Suppose $\mathcal{S}=\langle\lang,\vdash\rangle$ is a logic. pfECQ1 holds in $\mathcal{S}$ iff pfECQ2 holds in it, i.e., pfECQ1 is equivalent to pfECQ2.
\end{thm}

\begin{proof}
    Suppose pfECQ1 holds in $\mathcal{S}$. Let $\Gamma\subsetneq\lang$. Then, by pfECQ1, there exists $\Delta\subsetneq\lang$ such that $\Gamma\subseteq\Delta$ and $C_\vdash(\Delta)=\lang$. Since $\Gamma\subseteq\Delta$, $\Gamma\cup\Delta=\Delta$. So, $\Gamma\cup\Delta\subsetneq\lang$ and $C_\vdash(\Gamma\cup\Delta)=\lang$. Thus, pfECQ2 holds.

    Conversely, suppose pfECQ2 holds in $\mathcal{S}$. Let $\Gamma\subsetneq\lang$. Then, by pfECQ2, there exists $\Delta\subsetneq\lang$ such that $\Gamma\cup\Delta\subsetneq\lang$ and $C_\vdash(\Gamma\cup\Delta)=\lang$. Now, $\Gamma\subseteq\Gamma\cup\Delta$. Thus, pfECQ1 holds in $\mathcal{S}$.
\end{proof}

\begin{rem}
    In light of the above theorem, we will use pfECQ to refer to both pfECQ1 and pfECQ2.
\end{rem}

\begin{thm}\label{thm:spECQ->pfECQ}
    Suppose $\mathcal{S}=\langle\lang,\vdash\rangle$ is a logic. If spECQ holds in $\mathcal{S}$, then pfECQ also holds in it.
\end{thm}

\begin{proof}
    Suppose spECQ holds in $\mathcal{S}$. Let $\Gamma\subsetneq\lang$. Then, by spECQ, there exists $\alpha\in\lang$ such that $\Gamma\cup\{\alpha\}\subsetneq\lang$ and $C_\vdash(\Gamma\cup\{\alpha\})=\lang$. Using $\Delta=\Gamma\cup\{\alpha\}$, we see that pfECQ holds in $\mathcal{S}$.
\end{proof}

\begin{thm}\label{thm:pfECQ->sECQ}
    Suppose $\mathcal{S}=\langle\lang,\vdash\rangle$ is a logic such that $\lang$ has at least two distinct elements. If pfECQ holds in $\mathcal{S}$, then sECQ also holds in it.
\end{thm}

\begin{proof}
    Suppose pfECQ holds in $\mathcal{S}$. Let $\alpha\in\lang$. Since $\lang$ has at least two distinct elements, $\{\alpha\}\subsetneq\lang$. Then, by pfECQ, there exists $\Delta\subsetneq\lang$ such that $\{\alpha\}\subseteq\Delta$, i.e., $\alpha\in\Delta$, and $C_\vdash(\Delta)=\lang$. Thus, sECQ holds in $\mathcal{S}$.
\end{proof}

The following example shows that the converse of Theorem \ref{thm:spECQ->pfECQ} does not hold, i.e., pfECQ does not imply spECQ. It also shows that neither pfECQ nor sECQ implies gECQ.

\begin{exa}\label{exa:pfECQ1!->spECQ}
    Let $\mathcal{S}=\langle\lang,\vdash\rangle$ be a logic, where $\lang$ is infinite, and $\vdash\,\subseteq\pow(\lang)\times\lang$ is such that, for any $\Gamma\subseteq\lang$,
    \[
    C_\vdash(\Gamma)=\left\{\begin{array}{ll}
         \lang&\hbox{if }\Gamma\hbox{ is infinite}, \\
         \Gamma&\hbox{otherwise}.
    \end{array}\right.
    \]
    We first show that pfECQ holds in $\mathcal{S}$. Let $\Gamma\subsetneq\lang$. If $\Gamma$ is infinite, then we are done since $C_\vdash(\Gamma)=\lang$. On the other hand, if $\Gamma$ is finite, then as $\lang$ is infinite, there exists $\alpha\in\lang\setminus\Gamma$. Let $\Delta=\lang\setminus\{\alpha\}$. Clearly, $\Gamma\subseteq\Delta\subsetneq\lang$ and $\Delta$ is infinite, which implies that $C_\vdash(\Delta)=\lang$. Thus, for every $\Gamma\subsetneq\lang$, there exists $\Delta\subsetneq\lang$ such that $\Gamma\subseteq\Delta$, and $C_\vdash(\Delta)=\lang$, i.e., pfECQ holds in $\mathcal{S}$. Then, as $\lang$ is infinite, and so, has at least two distinct elements, by Theorem \ref{thm:pfECQ->sECQ}, sECQ also holds in $\mathcal{S}$.

    Now, to see that spECQ does not hold in $\mathcal{S}$, let $\Gamma$ be any finite subset of $\lang$. Since $\lang$ is infinite, $\Gamma\subsetneq\lang$. Now, for any $\alpha\in\lang$, $\Gamma\cup\{\alpha\}$ is also finite, and hence, $C_\vdash(\Gamma\cup\{\alpha\})=\Gamma\cup\{\alpha\}\neq\lang$. Thus, spECQ fails in $\mathcal{S}$.

    Finally, for any $\alpha,\beta\in\lang$, since $\{\alpha,\beta\}$ is finite, $C_\vdash(\{\alpha,\beta\})=\{\alpha,\beta\}\subsetneq\lang$. Thus, gECQ cannot hold in $\mathcal{S}$.
\end{exa}

The following example shows that neither gECQ nor sECQ implies pfECQ or spECQ.

\begin{exa}\label{exa:gECQ!->pfECQ1}
    Let $\mathcal{S}=\langle\lang,\vdash\rangle$ be the logic, where $\lang=\NN$, and $\vdash\,\subseteq\pow(\lang)\times\lang$ is such that, for any $\Gamma\subseteq\lang$,
    \[
    C_\vdash(\Gamma)=\left\{\begin{array}{ll}
         \lang&\hbox{if }\Gamma=\{n,n+1\}\hbox{ for some }n\in\NN,\\
         \Gamma&\hbox{otherwise}. 
    \end{array}\right.
    \]
    Let $k\in\lang$. Then, $k+1\in\lang$ and $C_\vdash(\{k,k+1\})=\lang$. Thus, gECQ holds in $\mathcal{S}$. Since gECQ holds in $\mathcal{S}$, and $\lang$ has at least three distinct elements, by Theorem \ref{thm:gECQ->sECQ}, sECQ also holds in $\mathcal{S}$.

    Now, to see that pfECQ does not hold in $\mathcal{S}$, let $\Gamma\subsetneq\lang$ such that $\lvert\Gamma\rvert\ge3$. Then, for any $\Delta\subseteq\lang$ such that $\Gamma\subseteq\Delta$, $C_\vdash(\Delta)=\Delta$. So, for any $\Delta\supseteq\Gamma$, $C_\vdash(\Delta)=\lang$ iff $\Delta=\lang$. Thus, there does not exist any $\Delta\subsetneq\lang$ such that $\Gamma\subseteq\Delta$ and $C_\vdash(\Delta)=\lang$. Hence, pfECQ fails in $\mathcal{S}$.

    By similar arguments as above, we can say that for any $\Gamma\subsetneq\lang$ with $\lvert\Gamma\rvert\ge3$, there does not exist any $\alpha\in\lang$ such that $\Gamma\cup\{\alpha\}\subsetneq\lang$ and $C_\vdash(\Gamma\cup\{\alpha\})=\lang$. Thus, spECQ fails in $\mathcal{S}$.
\end{exa}

\begin{rem}
    The fact that gECQ does not imply spECQ can also be derived from the fact that gECQ does not imply pfECQ (shown in the above example) and Theorem \ref{thm:spECQ->pfECQ}, which states that spECQ implies pfECQ.
\end{rem}

\begin{thm}\label{thm:pfECQ3<->spECQ}
    Suppose $\mathcal{S}=\langle\lang,\vdash\rangle$ is a logic. If pfECQ3 holds in $\mathcal{S}$, then spECQ holds in it. The converse also holds if $\lang$ has at least two distinct elements. Thus, if $\lang$ has at least two distinct elements, pfECQ3 is equivalent to spECQ. 
\end{thm}

\begin{proof}
    Suppose pfECQ3 holds in $\mathcal{S}$. Let $\Gamma\subsetneq\lang$. Then, by pfECQ3, there exists $\emptyset\neq\Delta\subsetneq\lang$ such that $\Gamma\cup\Delta\subsetneq\lang$, and for every $\emptyset\neq\Delta^\prime\subseteq\Delta$, $C_\vdash(\Gamma\cup\Delta^\prime)=\lang$. Since $\Delta\neq\emptyset$, there exists $\alpha\in\Delta$. Moreover, $\Gamma\cup\{\alpha\}\subseteq\Gamma\cup\Delta\subsetneq\lang$. Finally, since $\emptyset\neq\{\alpha\}\subseteq\Delta$, $C_\vdash(\Gamma\cup\{\alpha\})=\lang$. Thus, for every $\Gamma\subsetneq\lang$, there exists $\alpha\in\lang$ such that $\Gamma\cup\{\alpha\}\subsetneq\lang$, and $C_\vdash(\Gamma\cup\{\alpha\})=\lang$, i.e., spECQ holds in $\mathcal{S}$.

    Conversely, suppose $\lang$ has at least two distinct elements, and spECQ holds in $\mathcal{S}$. Let $\Gamma\subsetneq\lang$. Then, by spECQ, there exists $\alpha\in\lang$ such that $\Gamma\cup\{\alpha\}\subsetneq\lang$ and $C_\vdash(\Gamma\cup\{\alpha\})=\lang$. Now, let $\Delta=\{\alpha\}$. Since $\lang$ has at least two distinct elements, $\emptyset\neq\Delta\subsetneq\lang$. Moreover, $\Gamma\cup\Delta=\Gamma\cup\{\alpha\}\subsetneq\lang$. Finally, since $\Delta$ is the only non-empty subset of itself and $C_\vdash(\Gamma\cup\Delta)=C_\vdash(\Gamma\cup\{\alpha\})=\lang$, we have that for every $\emptyset\neq\Delta^\prime\subseteq\Delta$, $C_\vdash(\Gamma\cup\Delta^\prime)=\lang$. Thus, pfECQ3 holds in $\mathcal{S}$.
\end{proof}

\begin{cor}\label{cor:pfECQ3->pfECQ/sECQ}
    Suppose $\mathcal{S}=\langle\lang,\vdash\rangle$ is a logic. If pfECQ3 holds in $\mathcal{S}$, then pfECQ holds in $\mathcal{S}$. If $\lang$ has at least two distinct elements, then sECQ also holds in $\mathcal{S}$.
\end{cor}

\begin{proof}
    Suppose pfECQ3 holds in $\mathcal{S}$. Then, by the above theorem, spECQ holds in $\mathcal{S}$. So, by Theorem \ref{thm:spECQ->pfECQ}, pfECQ holds in $\mathcal{S}$. 
    
    Moreover, if $\lang$ has at least three distinct elements, then, by Theorem \ref{thm:spECQ->sECQ}, since spECQ holds in $\mathcal{S}$, sECQ also holds in it.
\end{proof}

\begin{cor}
    We can now draw the following conclusions from the results obtained so far.
    \begin{enumerate}[label=(\roman*)]
        \item sECQ does not imply pfECQ3.
        \item pfECQ does not imply pfECQ3.
    \end{enumerate}
\end{cor}

\begin{proof}
    \begin{enumerate}[label=(\roman*)]
        \item Suppose the contrary, i.e., sECQ implies pfECQ3. Now, by Theorem \ref{thm:pfECQ3<->spECQ}, pfECQ3 implies spECQ. So, sECQ implies spECQ. However, by Example \ref{exa:sECQ!->spECQ}, we know that that is not the case. Hence, sECQ cannot imply pfECQ3.

        \item Suppose the contrary, i.e., pfECQ implies pfECQ3. Now, by Theorem \ref{thm:pfECQ3<->spECQ}, pfECQ3 implies spECQ. So, pfECQ implies spECQ. However, by Example \ref{exa:pfECQ1!->spECQ}, we know that that is not the case. Hence, pfECQ does not imply pfECQ3. 
    \end{enumerate}
\end{proof}

We have already seen, via Examples \ref{exa:PWK-NotNFpara}, \ref{exa:sECQ-P1}, and Remarks \ref{rem:GentlyParaconsistent-sECQ}, and \ref{rem:Pm-sECQ}, that neither gECQ nor sECQ implies ECQ. The following example shows that none of the other explosion principles, pfECQ, pfECQ3 or spECQ implies ECQ.

\begin{exa}\label{exa:spECQ!->ECQ}
    Let $\mathcal{S}=\langle\lang,\vdash\rangle$ be the logic PWK (as described in \cite{Bonzio2017}). We know that PWK is paraconsistent, and so ECQ fails in it. As discussed in Example \ref{exa:PWK-NotNFpara}, PWK is not $\bot$-paraconsistent as there is a falsity constant $\bot\in\lang$. We claim that spECQ holds in PWK. Let $\Gamma\subsetneq\lang$. If $C_\vdash(\Gamma)=\lang$, then $\Gamma\neq\emptyset$ since not every formula is a theorem of PWK. So, for any $\alpha\in\Gamma$, $\Gamma\cup\{\alpha\}=\Gamma\subsetneq\lang$ and $C_\vdash(\Gamma\cup\{\alpha\})=C_\vdash(\Gamma)=\lang$. On the other hand, if $C_\vdash(\Gamma)\neq\lang$, then $C_\vdash(\Gamma\cup\{\bot\})=\lang$. Now, suppose $\Gamma\cup\{\bot\}=\lang$, then $\Gamma=\lang\setminus\{\bot\}$, which implies that $\bot\land\bot\in\Gamma$. However, as $C_\vdash(\bot\land\bot)=\lang$, by monotonicity, $C_\vdash(\Gamma)=\lang$, a contradiction to our previous assumption. Thus, $\Gamma\cup\{\bot\}\subsetneq\lang$. Hence, we see that for any $\Gamma\subsetneq\lang$, there exists $\alpha\in\lang$ such that $\Gamma\cup\{\alpha\}\subsetneq\lang$ and $C_\vdash(\Gamma\cup\{\alpha\})=\lang$, i.e., spECQ holds in PWK.

    Since spECQ holds in PWK, pfECQ holds in it, by Theorem \ref{thm:spECQ->pfECQ}. Moreover, since $\lvert\lang\rvert\ge2$, pfECQ3 also holds in it, by Theorem \ref{thm:pfECQ3<->spECQ}. 
\end{exa}

\begin{exa}\label{exa:spECQ/pfECQ-P1}
    The logic $P^1$, discussed in Example \ref{exa:sECQ-P1}, is also an example of a logic where spECQ holds. This can be shown as follows. Let $P^1$ $=\langle\lang,\vdash\rangle$, as before, and $\Gamma\subsetneq\lang$. Now, as argued in Example \ref{exa:sECQ-P1}, $C_\vdash(\neg(\alpha\limp\alpha))=\lang$. Since $P^1$ is monotonic, $C_\vdash(\Gamma\cup\{\neg(\alpha\limp\alpha)\})=\lang$. Now, if $\Gamma\cup\{\neg(\alpha\limp\alpha)\}\subsetneq\lang$, then we are done. If not, then $\Gamma=\lang\setminus\{\neg(\alpha\limp\alpha)\}$, and hence, $\neg((\alpha\limp\alpha)\limp(\alpha\limp\alpha))\in\Gamma$. Now, $C_\vdash(\neg((\alpha\limp\alpha)\limp(\alpha\limp\alpha)))=\lang$ as well, and so, by monotonicity, $C_\vdash(\Gamma\cup\{\neg((\alpha\limp\alpha)\limp(\alpha\limp\alpha))\})=\lang$ and $\Gamma\cup\{\neg((\alpha\limp\alpha)\limp(\alpha\limp\alpha))\}=\Gamma\subsetneq\lang$. Thus, spECQ holds in $P^1$.

    Therefore, by Theorem \ref{thm:spECQ->pfECQ}, pfECQ also holds in $P^1$. Hence the logic $P^1$ is another example of a logic, where spECQ and pfECQ hold, but ECQ fails.
\end{exa}

\begin{rem}
    It can be shown using the same arguments, as in the above example, that spECQ, and hence, pfECQ, hold in all the gently paraconsistent logics from \cite{Ciucura2020-1} (discussed in Remark \ref{rem:GentlyParaconsistent-sECQ}) and the paraconsistent logics $P^m$ and $P^{*m}$, $m\ge1$, from \cite{Ciucura2020} (mentioned in Remark \ref{rem:Pm-sECQ}). This is because all of these logics are monotonic and contain the positive fragment of classical propositional logic. So, for any $\alpha$, as $\alpha\limp\alpha$ is a theorem in each of them, $C_\vdash(\neg(\alpha\limp\alpha))=\lang$.
\end{rem}

It has been established now that ECQ is not implied by any of the other principles of explosion. However, ECQ implies sECQ and gECQ. We argue below that ECQ does not imply the two remaining explosion principles, viz., spECQ and pfECQ.

\begin{thm}\label{thm:spECQ->F}
   Suppose $\mathcal{S}=\langle\lang,\vdash\rangle$ is a logic. If spECQ holds in $\mathcal{S}$, then $\lang$ has at least two distinct elements, and there exists $\alpha\in\lang$ such that $C_\vdash(\alpha)=\lang$. 
\end{thm}

\begin{proof}
    Suppose spECQ holds in $\mathcal{S}$. Now, $\emptyset\subsetneq\lang$. Then, by spECQ, there exists $\alpha\in\lang$ such that $\emptyset\cup\{\alpha\}=\{\alpha\}\subsetneq\lang$, which implies that $\lang$ has at least two distinct elements, and $C_\vdash(\alpha)=\lang$.
\end{proof}

\begin{rem}
    We can conclude from the above theorem that if $\mathcal{S}=\langle\lang,\vdash\rangle$ is a logic such that $\neg$-ECQ holds, where $\neg$ is a unary operator in $\mathcal{S}$, i.e., $\mathcal{S}$ is not paraconsistent with respect to $\neg$, but there does not exist $\alpha\in\lang$ such that $C_\vdash(\alpha)=\lang$, which, in particular, implies that there is no falsity constant, then spECQ fails in $\mathcal{S}$. We construct such a logic below.
\end{rem}

\begin{exa}\label{exa:ECQ!->spECQ/pfECQ}
    Let $\mathcal{S}=\langle\lang,\vdash\rangle$ be a logic, where $\lang=\ZZ$, and $\vdash\,\subseteq\pow(\lang)\times\lang$ is such that, for any $\Gamma\subseteq\lang$,
    \[
    C_\vdash(\Gamma)=\left\{\begin{array}{ll}
         \lang&\hbox{if }\Gamma\neq\emptyset\hbox{ and for each }n\in\Gamma,-n\in\Gamma,\\
         \Gamma&\hbox{otherwise}.
    \end{array}\right.
    \]
    Then, with respect to the unary operator $f:\ZZ\to\ZZ$ defined by $f(n)=-n$ for each $n\in\ZZ$, we see that ECQ holds, i.e., $f$-ECQ holds, since for each $n\in\ZZ$, $C_\vdash(\{n,-n\})=\lang$. 
    
    To see that spECQ does not hold in $\mathcal{S}$, let $\Gamma=\{n\in\ZZ\mid\,n\ge0\}\subsetneq\lang$. Clearly, for any $n\in\lang$, $\Gamma\cup\{n\}\subsetneq\lang$ but $C_\vdash(\Gamma\cup\{n\})\neq\lang$. 

    In fact, the only subset $\Delta\subseteq\lang$ such that $\Gamma\subseteq\Delta$ and $C_\vdash(\Delta)=\lang$ is $\ZZ$. Hence, there is no $\Delta\subsetneq\lang$ such that $\Gamma\subseteq\Delta$ and $C_\vdash(\Delta)=\lang$. So, pfECQ also fails in $\mathcal{S}$.
\end{exa}

Thus, ECQ does not imply spECQ or pfECQ. The following theorem, however, shows that, under some conditions, ECQ does imply spECQ and pfECQ.

\begin{thm}\label{thm:ECQ!->spECQ/pfECQ}
    Suppose $\mathcal{S}=\langle\lang,\vdash\rangle$ is a monotonic logic, where $\lang$ is the formula algebra over a denumerable set of variables $V$, using a signature containing a unary (negation) operator $\neg$. Moreover, suppose there exists $\alpha\in\lang$ such that $C_\vdash(\alpha)=\lang$. If $\neg$-ECQ holds in $\mathcal{S}$, then spECQ and pfECQ also hold in it.
\end{thm}

\begin{proof}
    Suppose $\neg$-ECQ holds in $\mathcal{S}$. Let $\Gamma\subsetneq\lang$. To show that spECQ holds, we need to find a $\varphi\in\lang$ such that $\Gamma\cup\{\varphi\}\subsetneq\lang$ and $C_\vdash(\Gamma\cup\{\varphi\})=\lang$. Now, by monotonicity, $C_\vdash(\Gamma\cup\{\alpha\})=\lang$. If $\Gamma\cup\{\alpha\}\subsetneq\lang$, then we are done. If not, then $\Gamma=\lang\setminus\{\alpha\}$. So, $\Gamma\neq\emptyset$ since $\lang$ is denumerable. Now, let $p\in V$ such that $p$ does not occur in $\alpha$. Then $p,\neg p\neq\alpha$, and hence, $p,\neg p\in\Gamma$. Now, since $\neg$-ECQ holds in $\mathcal{S}$, $C_\vdash(\{p,\neg p\})=\lang$. Thus, $\Gamma\cup\{p\}=\Gamma\subsetneq\lang$, and, by monotonicity, $C_\vdash(\Gamma\cup\{p\})=C_\vdash(\Gamma)=\lang$. So, spECQ holds in $\mathcal{S}$. Then, by Theorem \ref{thm:spECQ->pfECQ}, pfECQ also holds in $\mathcal{S}$.
\end{proof}

The connections between the different explosion principles discussed so far can now be summarized in the following figure. The arrows in the figure indicate implications, while the absence of an arrow from one to another principle of explosion indicates a non-implication. Some implications hold under mild conditions, such as that $\lang$ must have at least two or three elements. We have not shown these conditions in the picture to preserve clarity.

\begin{figure}[ht]
    \centering
    \[\begin{tikzcd}
	&& {\mathrm{sECQ}^\prime} \\
	&& {\mathrm{sECQ}} & {\mathrm{ECQ}} \\
	{\mathrm{pfECQ3}} & {\mathrm{spECQ}} && {\mathrm{gECQ}} \\
	&& {\begin{array}{c}\mathrm{pfECQ}\\(\mathrm{pfECQ1}\longleftrightarrow\mathrm{pfECQ2})\end{array}}
	\arrow[from=1-3, to=2-3]
	\arrow[from=2-3, to=1-3]
	\arrow[from=3-2, to=2-3]
	\arrow[from=2-4, to=2-3]
	\arrow[from=3-4, to=2-3]
	\arrow[from=3-2, to=3-4]
	\arrow[from=4-3, to=2-3]
	\arrow[from=2-4, to=3-4]
    \arrow[from=3-2, to=4-3]
	\arrow[from=3-1, to=3-2]
	\arrow[from=3-2, to=3-1]
\end{tikzcd}\]
    \caption{Principles of Explosion}
    \label{fig:Explosion}
\end{figure}
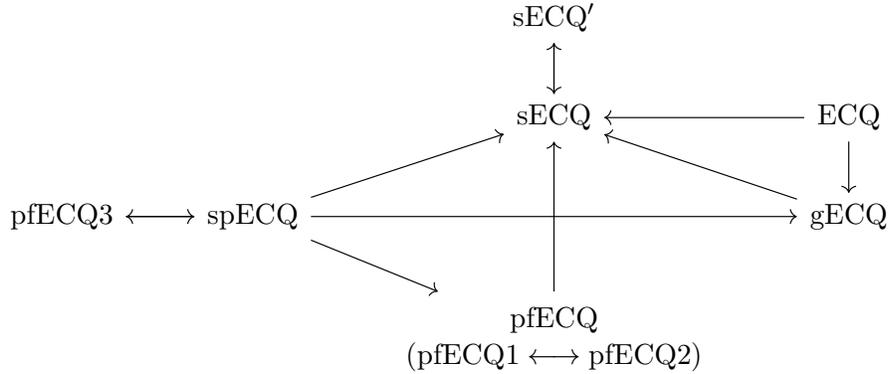

We next explore the connections between the various principles of explosion discussed in this section with the property of \emph{finite trivializability}. 

\begin{dfn}\label{dfn:fintriv}
    Suppose $\mathcal{S}=\langle\lang,\vdash\rangle$ is a logic. $\mathcal{S}$ is said to be \emph{finitely trivializable} if there exists a finite $\Gamma\subseteq\lang$ such that $C_\vdash(\Gamma)=\lang$.
\end{dfn}

\begin{rem}\label{rem:fintriv}
    It is clear that if $\lang$ is finite, then finite trivializability is equivalent to trivializability. Thus, it makes sense to investigate this property in the case of logics, where $\lang$ is infinite. Hence, $\lang$ will be assumed to be infinite in all the following results concerning finite trivializability.
\end{rem}

It is fairly easy to see that finite trivializability does not necessarily imply any of the principles of explosion: ECQ, gECQ, sECQ, spECQ, pfECQ. The following results expound the conditions under which finite trivializability does imply gECQ, sECQ, spECQ, and pfECQ.

\begin{thm}\label{thm:fintriv-sECQ}
    Suppose $\mathcal{S}=\langle\lang,\vdash\rangle$ is a monotonic logic. If $\mathcal{S}$ is finitely trivializable, then sECQ holds in $\mathcal{S}$.
\end{thm}

\begin{proof}
    Suppose $\mathcal{S}$ is finitely trivializable. Then, there exists a finite $\Gamma\subseteq\lang$ such that $C_\vdash(\Gamma)=\lang$. Let $\alpha\in\lang$. Then $\Gamma\cup\{\alpha\}$ is a finite subset of $\lang$, and hence $\Gamma\cup\{\alpha\}\subsetneq\lang$ since $\lang$ is infinite. Clearly, $\alpha\in\Gamma\cup\{\alpha\}$, and, by monotonicity, $C_\vdash(\Gamma\cup\{\alpha\})=\lang$. Thus, sECQ holds in $\mathcal{S}$.
\end{proof}

\begin{lem}\label{lem:fintriv-F}
    Suppose $\mathcal{S}=\langle\lang,\vdash\rangle$ is a transitive logic, where $\lang$ is the formula algebra over a set of variables $V\neq\emptyset$, using a finite signature containing a conjunction operator $\land$. Moreover, suppose the following elimination rule for $\land$ holds in $\mathcal{S}$. 
    \[
    \alpha\land\beta\vdash\alpha\;\hbox{ and }\;\alpha\land\beta\vdash\beta\quad(\land E).
    \]
    Then, $\mathcal{S}$ is finitely trivializable  iff there exists $\varphi\in\lang$ such that $C_\vdash(\varphi)=\lang$.
\end{lem}

\begin{proof}
    Suppose $\mathcal{S}$ is finitely trivializable. Let $\Gamma$ be a finite subset of $\lang$ such that $C_\vdash(\Gamma)=\lang$, and $\varphi=\displaystyle\bigwedge_{\beta\in\Gamma}\beta$ (the conjunction of all the elements in the finite set $\Gamma$). Now, by $\land E$, $\varphi\vdash\beta$ for all $\beta\in\Gamma$. Since $C_\vdash(\Gamma)=\lang$, by transitivity, $C_\vdash(\varphi)=\lang$. The converse clearly holds since $\{\varphi\}$ is a finite set.
\end{proof}

\begin{thm}\label{thm:fintriv-spECQ}
    Suppose $\mathcal{S}=\langle\lang,\vdash\rangle$ is a monotonic and transitive logic, where $\lang$ is as in the above lemma. Moreover, suppose the elimination rule $(\land E)$, as described in the above lemma, holds in $\mathcal{S}$. If $\mathcal{S}$ is finitely trivializable, then spECQ, pfECQ, and gECQ hold in $\mathcal{S}$.
\end{thm}

\begin{proof}
    Suppose $\mathcal{S}$ is finitely trivializable. So, by Lemma \ref{lem:fintriv-F}, there exists $\varphi\in\lang$ such that $C_\vdash(\varphi)=\lang$. Let $\Gamma\subsetneq\lang$. 
    
    \textsc{Case 1:} $\Gamma\cup\{\varphi\}\subsetneq\lang$. By monotonicity, $C_\vdash(\Gamma\cup\{\varphi\})=\lang$.

    \textsc{Case 2:} $\Gamma\cup\{\varphi\}=\lang$. So, $\Gamma=\lang\setminus\{\varphi\}$, which implies that $\varphi\land\varphi\in\Gamma$ since $\varphi\neq\varphi\land\varphi$. Now, by $(\land E)$, $\varphi\land\varphi\vdash\varphi$. So, by monotonicity, $\Gamma\cup\{\varphi\land\varphi\}\vdash\varphi$ (i.e., $\Gamma\vdash\varphi$). Then, as $C_\vdash(\varphi)=\lang$, by transitivity, $C_\vdash(\Gamma\cup\{\varphi\land\varphi\})=C_\vdash(\Gamma)=\lang$.

    Thus, spECQ holds in $\mathcal{S}$. So, by Theorem \ref{thm:spECQ->pfECQ}, pfECQ also holds in $\mathcal{S}$. Since $\lang$ has at least two distinct elements (in fact, $\lang$ is denumerable as it is generated inductively over a non-empty set), by Theorem \ref{thm:spECQ->sECQ}, gECQ also holds in $\mathcal{S}$.
\end{proof}

\begin{rem}\label{rem:non-fintriv}
    Suppose $\mathcal{S}=\langle\lang,\vdash\rangle$ such that $\lang$ is infinite. It is easy to see that if ECQ or gECQ holds in $\mathcal{S}$, then it is finitely trivializable. Moreover, by Theorem \ref{thm:spECQ->F}, if spECQ holds in $\mathcal{S}$, then there exists $\varphi\in\lang$ such that $C_\vdash(\varphi)=\lang$, which implies that $\mathcal{S}$ is finitely trivializable. In other words, if $\mathcal{S}$ is not finitely trivializable, then none of the explosion principles ECQ, gECQ, or spECQ holds in it. The 3-valued logic $LP$ (\emph{Logic of Paradox}), introduced in \cite{Asenjo1966} and investigated in \cite{Priest1979}, is one such logic. An interesting feature of $LP$ is that it does not have a detachable implication. The logic $Pac$, obtained by adding a detachable implication to $LP$, is also not finitely trivializable. This is also a 3-valued paraconsistent logic that first appeared with this name in \cite{Avron1991}, as $RM_3^\supset$ in \cite{Avron1986}, and as $PI^s$ in \cite{Batens1980}. The logic $PI$, introduced in \cite{Batens1980}, is also not finitely trivializable. (See also the discussion on $Pac$ and $PI$ in \cite{Carnielli2007}.)
    
    It may, however, be noted that sECQ and pfECQ may or may not hold in logics that are not finitely trivializable. Clearly, if there does not exist any $\Gamma\subsetneq\lang$ such that $C_\vdash(\Gamma)=\lang$, then sECQ and pfECQ also fail. The following example shows that, under some conditions, pfECQ and sECQ might hold in logics that are not finitely trivializable, provided there is an infinite proper subset of $\lang$ that explodes.
\end{rem}

\begin{exa}\label{exa:nonfintriv-sECQ/pfECQ}
    Suppose $\mathcal{S}=\langle\lang,\vdash\rangle$ is a logic, where $\lang$ is infinite, and $\vdash\,\subseteq\pow(\lang)\times\lang$ is such that, for any $\Gamma\subseteq\lang$,
    \[
    C_\vdash(\Gamma)=\left\{\begin{array}{ll}
         \Gamma&\hbox{if }\Gamma\hbox{ is finite},  \\
         \lang&\hbox{otherwise}. 
    \end{array}\right.
    \]
    Clearly, $\mathcal{S}$ is not finitely trivializable. Let $\Gamma\subsetneq\lang$. If $\Gamma$ is infinite, then $C_\vdash(\Gamma)=\lang$. On the other hand, if $\Gamma$ is finite, let $\alpha\in\lang\setminus\Gamma$ (such an $\alpha$ exists since $\Gamma$ is finite, while $\lang$ is infinite). Now, let $\Delta=\lang\setminus\{\alpha\}$. Then, $\Gamma\subseteq\Delta\subsetneq\lang$, and since $\Delta$ is infinite, $C_\vdash(\Delta)=\lang$. Thus, pfECQ holds in $\mathcal{S}$, and so, by Theorem \ref{thm:pfECQ->sECQ}, sECQ also holds in $\mathcal{S}$.
\end{exa}

We end the section with a class of logics that might serve as important examples to motivate one to consider the above principles of explosion. These are the \emph{Logics of Formal Inconsistency (LFIs)} \cite{CarnielliConiglio2016, Carnielli2007}. For these logics, we assume that $\lang$ is the formula algebra over a denumerable set of variables $V$, using a finite signature containing a unary (negation) operator $\neg$. These logics are also assumed to be Tarskian, structural, and finitary, i.e., for any $\Gamma\cup\{\alpha\}\subseteq\lang$, if $\Gamma\vdash\alpha$, then there exists a finite $\Gamma^\prime\subseteq\Gamma$ such that $\Gamma^\prime\vdash\alpha$. A Tarskian, structural, finitary logic is sometimes referred to as a \emph{standard} logic. An LFI is a paraconsistent logic, i.e., ECQ fails in it, where a local or controlled explosion is allowed. It is important to note that not all paraconsistent logics are LFIs, and examples of such logics can be found in the works cited at the beginning of this paragraph. The following definition of an LFI is adapted from \cite{CarnielliConiglio2016}.

\begin{dfn}\label{dfn:LFI}
    Suppose $\mathcal{S}=\langle\lang,\vdash\rangle$ is a standard logic. Let $p\in V$ and $\bigcirc(p)$ be a non-empty set of formulas depending exactly on $p$. Moreover, for any $\varphi\in\lang$, let $\bigcirc(\varphi)=\{\psi(\varphi)\mid\,\psi(p)\in\bigcirc(p)\}$. $\mathcal{S}$ is a \emph{Logic of Formal Inconsistency (LFI)} (with respect to $\neg$ and $\bigcirc(p)$) if the following conditions hold.
    \begin{enumerate}[label=(\roman*)]
        \item There exists $\varphi\in\lang$ such that $C_\vdash(\{\varphi,\neg\varphi\})\neq\lang$, i.e., $\neg$-ECQ fails in $\mathcal{S}$, or equivalently, $\mathcal{S}$ is paraconsistent with respect to $\neg$.
        \item There exists $\alpha,\beta\in\lang$ such that $\beta\notin C_\vdash(\bigcirc(\alpha)\cup\{\alpha\})\cup C_\vdash(\bigcirc(\alpha)\cup\{\neg\alpha\})$, which implies that $C_\vdash(\bigcirc(\alpha)\cup\{\alpha\})\neq\lang$ and $C_\vdash(\bigcirc(\alpha)\cup\{\neg\alpha\})\neq\lang$.
        \item $C_\vdash(\bigcirc(\varphi)\cup\{\varphi,\neg\varphi\})=\lang$ for all $\varphi\in\lang$.
    \end{enumerate}
\end{dfn}

There are two other ways of introducing LFIs. These are called \emph{weak} and \emph{strong} LFIs. It follows easily from the definitions of LFIs, weak LFIs, and strong LFIs that every strong LFI is an LFI, and every LFI is a weak LFI. The converses are not true. A detailed discussion on these can be found in \cite{CarnielliConiglio2016, Carnielli2007}.

Given an $\alpha\in\lang$, the set $\bigcirc(\alpha)$ is intended to express, in a sense, the consistency of $\alpha$ relative to the logic $\mathcal{S}$. When this set is a singleton, its element is denoted by $\circ\alpha$, and in this case, $\circ$ defines a \emph{consistency connective} or \emph{consistency operator}. The consistency connective may or may not be a primitive connective of the logic. 

Clause (iii) in the above definition is called the \emph{gentle principle of explosion} (not to be confused with the principle of gentle explosion discussed in Remark \ref{rem:GentlyParaconsistent-sECQ}), and a logic satisfying this clause is called \emph{gently explosive} (with respect to the respective $\bigcirc(p)$). If a logic is gently explosive with respect to some finite set $\bigcirc(p)$, it is called \emph{finitely gently explosive}. 

\begin{lem}\label{lem:fingentexpl->fintriv}
    Suppose $\mathcal{S}=\langle\lang,\vdash\rangle$ is an LFI with respect to $\neg$ and $\bigcirc(p)$. If $\mathcal{S}$ is finitely gently explosive, i.e., $\bigcirc(p)$ is finite, then it is finitely trivializable. Moreover, suppose the signature of $\mathcal{S}$ has a conjunction operator $\land$ for which the elimination rule $(\land E)$, as described in Lemma \ref{lem:fintriv-F}, holds\footnote{A logic with such a conjunction is called \emph{left-adjunctive} in \cite{Carnielli2007}.}. Then there exists $\varphi\in\lang$ such that $C_\vdash(\varphi)=\lang$\footnote{Such an element $\varphi$ is called a \emph{bottom particle} in \cite{Carnielli2007}, and a \emph{bottom/bottom formula} in \cite{CarnielliConiglio2016}.}.
\end{lem}

\begin{proof}
    Suppose $\mathcal{S}$ is finitely gently explosive. Then, for any $\alpha\in\lang$, $\bigcirc(\alpha)$ is finite, and hence, $\bigcirc(\alpha)\cup\{\alpha,\neg\alpha\}$ is a finite subset of $\lang$. Since $C_\vdash(\bigcirc(\alpha)\cup\{\alpha,\neg\alpha\})=\lang$, $\mathcal{S}$ is finitely trivializable.

    Now, in case the signature of $\mathcal{S}$ has a conjunction operator $\land$ such that $(\land E)$ holds, then, by Lemma \ref{lem:fintriv-F}, there exists $\varphi\in\lang$ such that $C_\vdash(\varphi)=\lang$.
\end{proof}

\begin{rem}
    It cannot, however, be concluded that a finitely trivializable LFI is always finitely gently explosive since the definition of an LFI does not rule out the presence of some $\alpha\in\lang$ such that $C_\vdash(\{\alpha,\neg\alpha\})=\lang$ even though $\neg$-ECQ does not hold.
\end{rem}

We already know that LFIs are paraconsistent, so $\neg$-ECQ fails in all of them. The following theorem lists the status of the other explosion principles in LFIs.

\begin{thm}\label{thm:LFI}
    Suppose $\mathcal{S}=\langle\lang,\vdash\rangle$ is an LFI with respect to $\neg$ and $\bigcirc(p)$.
    \begin{enumerate}[label=(\roman*)]
        \item pfECQ and sECQ hold in $\mathcal{S}$
        \item If $\mathcal{S}$ is finitely gently explosive, and the signature of $\mathcal{S}$ has a conjunction operator $\land$ for which the elimination rule $(\land E)$, as described in Lemma \ref{lem:fintriv-F}, holds, then spECQ and gECQ hold in it.
    \end{enumerate}
\end{thm}

\begin{proof}
    \begin{enumerate}[label=(\roman*)]
        \item Let $\Gamma\subsetneq\lang$. If $C_\vdash(\Gamma)=\lang$, then we are done. So, suppose $C_\vdash(\Gamma)\neq\lang$. Let $\alpha\in\lang$, and $\Delta=\Gamma\cup\bigcirc(\alpha)\cup\{\alpha,\neg\alpha\}$. Clearly, $\Gamma\subseteq\Delta$. Now, by the gentle principle of explosion, $C_\vdash(\bigcirc(\alpha)\cup\{\alpha,\neg\alpha\})=\lang$. So, by monotonicity, $C_\vdash(\Delta)=\lang$. 
        
        We claim that $\Delta\subsetneq\lang$. Suppose the contrary, i.e., $\Delta=\Gamma\cup\bigcirc(\alpha)\cup\{\alpha,\neg\alpha\}=\lang$. Now, $\bigcirc(\alpha)$ is the result of substituting $\alpha$ for the propositional variable $p$ in each element of $\bigcirc(p)$, which is a non-empty set of formulas depending exactly on $p$. So, $\var(\bigcirc(\alpha))=\var(\alpha)$, where $\var(\cdot)$ is as described in Definition \ref{dfn:lvarinclusion}. Thus, $\var(\bigcirc(\alpha)\cup\{\alpha,\neg\alpha\})=\var(\alpha)$, a finite set. Since the set of propositional variables $V$ is denumerable, there exists $q\in V\setminus\var(\bigcirc(\alpha)\cup\{\alpha,\neg\alpha\})$. Clearly, $(\bigcirc(q)\cup\{q,\neg q\})\cap(\bigcirc(\alpha)\cup\{\alpha,\neg\alpha\})=\emptyset$. Then, as $\Gamma\cup\bigcirc(\alpha)\cup\{\alpha,\neg\alpha\}=\lang$, $\bigcirc(q)\cup\{q,\neg q\}\subseteq\Gamma$. Again, by the gentle principle of explosion, $C_\vdash(\bigcirc(q)\cup\{q,\neg q\})=\lang$, and hence, by monotonicity, $C_\vdash(\Gamma)=\lang$. This is contrary to our assumption. So, $\Delta\subsetneq\lang$.
        
        Hence, for every $\Gamma\subsetneq\lang$, there exists $\Delta\subsetneq\lang$ such that $\Gamma\subseteq\Delta$ and $C_\vdash(\Delta)=\lang$. In other words, pfECQ holds in $\mathcal{S}$. Since $\lang$ is denumerable, and hence, $\lvert\lang\rvert\ge2$, by Theorem \ref{thm:pfECQ->sECQ}, sECQ also holds in $\mathcal{S}$.
        
        \item Suppose $\mathcal{S}$ is finitely gently explosive. Then, by Lemma \ref{lem:fingentexpl->fintriv}, $\mathcal{S}$ is finitely trivializable, and moreover, there exists $\varphi\in\lang$ such that $C_\vdash(\varphi)=\lang$. Since $\mathcal{S}$ is an LFI, it is monotonic and transitive. Hence, by Theorem \ref{thm:fintriv-spECQ}, spECQ and gECQ hold in $\mathcal{S}$.
    \end{enumerate}
\end{proof}

\begin{rem}
    We see, from the above theorem, that if $\mathcal{S}=\langle\lang,\vdash\rangle$ is a finitely gently explosive LFI with a conjunction operator $\land$ for which $(\land E)$ holds, then all the explosion principles discussed in this section, barring, of course, ECQ, hold in $\mathcal{S}$.

    The logic $P^1$ was discussed in Example \ref{exa:sECQ-P1}, and it was pointed out that while ECQ fails, gECQ and sECQ hold in it. Subsequently, it was also presented as a logic, where spECQ and pfECQ hold, in Example \ref{exa:spECQ/pfECQ-P1}. $P^1$ is, in fact, a finitely gently explosive LFI, as demonstrated in \cite{Carnielli2007}. Thus, the fact that all the principles of explosion, except for ECQ, hold in $P^1$ can be obtained as a corollary of Theorem \ref{thm:LFI} as well.

    If $\mathcal{S}$ is not finitely trivializable, and thus, not finitely gently explosive either, as discussed in Remark \ref{rem:non-fintriv}, none of ECQ, gECQ, and spECQ holds in $\mathcal{S}$. However, from the above theorem, we see that even in this case, sECQ and pfECQ, the only two principles of explosion that can hold in logics that are not finitely trivializable, hold in $\mathcal{S}$.
\end{rem}

\section{Principles of partial explosion}

In this section, we investigate some generalizations of gECQ, which differ from those considered in the previous section. The main idea is to try to move away from the universality of explosion. In other words, instead of having each element $\alpha$ of the underlying set $\lang$ of a logic $\mathcal{S}=\langle\lang,\vdash\rangle$ explode with some $\beta\in\lang$, the principles of explosion considered below limit the explosion to the members of some $\Gamma\subseteq\lang$ and that too with some or all elements of some $\Delta\subseteq\lang$. Hence, we call these \emph{principles of partial explosion}. Suppose $\mathcal{S}=\langle\lang,\vdash\rangle$ is a logic. Then, these principles can be formulated as follows.

\begin{itemize}
    \item There exists $\Gamma,\Delta\subseteq\lang$ such that $\Gamma\neq\emptyset$, and for every $\alpha\in\Gamma$, there exists $\beta\in\Delta$ such that $C_\vdash(\{\alpha,\beta\})=\lang$. (parECQ1)
    \item There exists $\Gamma,\Delta\subseteq\lang$ such that $\Gamma\times\Delta\neq\emptyset$, and for every $(\alpha,\beta)\in\Gamma\times\Delta$, $C_\vdash(\{\alpha,\beta\})=\lang$. (parECQ2)
\end{itemize}
The string `par' in the names of the above principles denote partial.

\begin{thm}\label{thm:parECQ1<->parECQ2}
    Suppose $\mathcal{S}=\langle\lang,\vdash\rangle$ be a logic. parECQ1 holds in $\mathcal{S}$ iff parECQ2 holds in it.
\end{thm}

\begin{proof}
    It is easy to see that parECQ2 implies parECQ1. To prove the converse, suppose parECQ1 holds in $\mathcal{S}$. Then, there exists $\Gamma,\Delta\subseteq\lang$ such that $\Gamma\neq\emptyset$, and for every $\alpha\in\Gamma$, there exists $\beta\in\Delta$ such that $C_\vdash(\{\alpha,\beta\})=\lang$. Now, let $\Gamma^\prime=\{\alpha\}\subseteq\Gamma$ and $\Delta^\prime=\{\beta\in\Delta\mid\,C_\vdash(\{\alpha,\beta\})=\lang\}$. Clearly, $\Gamma^\prime\neq\emptyset$ and, by parECQ1, $\Delta^\prime\neq\emptyset$. Hence, $\Gamma^\prime\times\Delta^\prime\neq\emptyset$. Now, by the construction of $\Gamma^\prime,\Delta^\prime$, for every $(\alpha,\beta)\in\Gamma^\prime\times\Delta^\prime$, $C_\vdash(\{\alpha,\beta\})=\lang$. Thus, parECQ2 holds.
\end{proof}

\begin{rem}
    In light of the above theorem, we will use parECQ to refer to both principles parECQ1 and parECQ2.
\end{rem}

\begin{thm}\label{thm:gECQ->parECQ}
    Suppose $\mathcal{S}=\langle\lang,\vdash\rangle$ is a logic. If gECQ holds in $\mathcal{S}$, then parECQ also holds in it.
\end{thm}

\begin{proof}
    Suppose gECQ holds in $\mathcal{S}$. Let $\alpha\in\lang$. Then, by gECQ, there exists $\beta\in\lang$ such that $C_\vdash(\{\alpha,\beta\})=\lang$. Let $\Gamma=\{\alpha\}$ and $\Delta=\{\beta\}$. This choice of $\Gamma,\Delta$ clearly shows that parECQ holds in $\mathcal{S}$.
\end{proof}

\begin{cor}
    Suppose $\mathcal{S}=\langle\lang,\vdash\rangle$ is a logic with a unary operator $\neg$. If $\neg$-ECQ holds in $\mathcal{S}$, then parECQ also holds in it.
\end{cor}

\begin{proof}
    Suppose $\neg$-ECQ holds in $\mathcal{S}$. Then, by Theorem \ref{thm:NFpara->para}, gECQ holds in $\mathcal{S}$. Thus, by Theorem \ref{thm:gECQ->parECQ}, parECQ holds in $\mathcal{S}$.
\end{proof}

\begin{cor}\label{cor:spECQ->parECQ}
    Suppose $\mathcal{S}=\langle\lang,\vdash\rangle$ is a logic such that $\lang$ has at least two distinct elements. If spECQ holds in $\mathcal{S}$, then parECQ also holds in it.
\end{cor}

\begin{proof}
    Suppose spECQ holds in $\mathcal{S}$. Since $\lang$ has at least two distinct elements, by Theorem \ref{thm:spECQ->sECQ}, gECQ holds in $\mathcal{S}$. Then, the conclusion follows from Theorem \ref{thm:gECQ->parECQ}.
\end{proof}

The converse of Theorem \ref{thm:gECQ->parECQ} does not hold as shown by the next example. 

\begin{exa}\label{exa:parECQ!->gECQ}
        Let $\mathcal{S}=\langle\lang,\vdash\rangle$ be a logic, where $\lang=\NN$ and $\vdash\,\subseteq\pow(\lang)\times\lang$ is such that, for any $\Gamma\subseteq\lang$,
    \[
    C_\vdash(\Gamma)=\left\{\begin{array}{ll}
         \lang&\hbox{if }\Gamma=\{0\},\\
         \Gamma&\hbox{otherwise}.
    \end{array}\right.
    \]
    Taking $\Gamma=\{0\}=\Delta$, we see that parECQ holds in $\mathcal{S}$. However, for any $\alpha\in\lang$ such that $\alpha\neq0$, $C_\vdash(\{\alpha,\beta\})\neq\lang$ for all $\beta\in\lang$. Thus, gECQ fails in $\mathcal{S}$.
\end{exa}

\begin{rem}
    We can construct similar examples, as above, to see that none of the other principles of explosion, viz., sECQ or pfECQ, implies or is implied by parECQ. However, that is unnecessary as it is not hard to see that parECQ is incomparable to sECQ and pfECQ. This is because sECQ and pfECQ yield an explosive set, of unknown size, for each element/subset of $\lang$, while parECQ claims the existence of two subsets $\Gamma,\Delta$ of $\lang$ such that for each $(\alpha,\beta)\in\Gamma\times\Delta$, the set $\{\alpha,\beta\}$ is explosive.
\end{rem}

Example \ref{exa:parECQ!->gECQ} already establishes that parECQ does not imply gECQ. Theorem \ref{thm:parECQ+!F->>gECQ} below further refines this non-implication by giving a condition under which parECQ implies the failure of gECQ. We first make the following definition for the ease of the subsequent presentation.

\begin{dfn}\label{dfn:complmentary}
    Suppose $\mathcal{S}=\langle\lang,\vdash\rangle$ is a logic. A pair of sets $(\Gamma,\Delta)\in\pow(\lang)\times\pow(\lang)$ is said to be 
    \begin{enumerate}[label=(\alph*)]
        \item \emph{complementary} in $\mathcal{S}$ if for every $(\alpha,\beta)\in\Gamma\times\Delta$, $C_\vdash(\{\alpha,\beta\})=\lang$.
        \item \emph{weakly complementary} in $\mathcal{S}$ if for every $\alpha\in\Gamma$, there exists $\beta\in\Delta$ such that $C_\vdash(\{\alpha,\beta\})=\lang$.
    \end{enumerate}
    
\end{dfn}

\begin{rem}\label{rem:parECQ-complementary}
     It is easy to see that, given a logic $\mathcal{S}=\langle\lang,\vdash\rangle$, parECQ (using the parECQ1 formulation) holds in $\mathcal{S}$ iff there is a weakly complementary pair of sets $(\Gamma,\Delta)\in\pow(\lang)\times\pow(\lang)$ in $\mathcal{S}$ with $\Gamma\neq\emptyset$. Also, parECQ (using the parECQ2 formulation) holds in $\mathcal{S}$ iff there is a complementary pair of sets $(\Gamma,\Delta)\in\pow(\lang)\times\pow(\lang)$ in $\mathcal{S}$ with $\Gamma\times\Delta\neq\emptyset$. Thus, there exists a complementary pair in $\mathcal{S}$ iff there is a weakly complementary pair in $\mathcal{S}$. If a pair of sets is complementary, then clearly, it is also weakly complementary. However, not every weakly complementary pair is complementary.
\end{rem}

\begin{lem}\label{lem:!F<->disjoint.complementary}
    Suppose $\mathcal{S}=\langle\lang,\vdash\rangle$ is a logic. There does not exist $\varphi\in\lang$ such that $C_\vdash(\varphi)=\lang$ iff for all complementary $(\Gamma,\Delta)$ in $\mathcal{S}$, with $\Gamma\times\Delta\neq\emptyset$, $\Gamma\cap\Delta=\emptyset$.
\end{lem}

\begin{proof}
    Suppose $(\Gamma,\Delta)\in\pow(\lang)\times\pow(\lang)$ is complementary in $\mathcal{S}$ such that $\Gamma\times\Delta\neq\emptyset$. Then, $\Gamma,\Delta\neq\emptyset$. Suppose further that $\Gamma\cap\Delta\neq\emptyset$, and let $\varphi\in\Gamma\cap\Delta$. Then, $(\varphi,\varphi)\in\Gamma\times\Delta$, and hence, by definition of a complementary pair of sets, $C_\vdash(\varphi)=\lang$. Thus, if there is no $\varphi\in\lang$ such that $C_\vdash(\varphi)=\lang$, then, for any complementary $(\Gamma,\Delta)$ in $\mathcal{S}$, with $\Gamma\times\Delta\neq\emptyset$, $\Gamma\cap\Delta=\emptyset$.

    Conversely, suppose that for any complementary $(\Gamma,\Delta)$ in $\mathcal{S}$, with $\Gamma\times\Delta\neq\emptyset$, $\Gamma\cap\Delta=\emptyset$. Let $\varphi\in\lang$ such that $C_\vdash(\varphi)=\lang$, and $\Gamma=\Delta=\{\varphi\}$. Then, clearly, $(\Gamma,\Delta)$ is complementary in $\mathcal{S}$ such that $\Gamma\times\Delta\neq\emptyset$ but $\Gamma\cap\Delta\neq\emptyset$. This contradicts our assumption regarding the complementary pairs in $\mathcal{S}$. Hence, there cannot be any $\varphi\in\lang$ such that $C_\vdash(\varphi)=\lang$.
\end{proof}

\begin{rem}
    It follows from the above lemma that, given a logic $\mathcal{S}=\langle\lang,\vdash\rangle$, if there exists $\varphi\in\lang$ such that $C_\vdash(\varphi)=\lang$, then there exists a complementary pair $(\Gamma,\Delta)$ in $\mathcal{S}$ with $\Gamma\times\Delta\neq\emptyset$. This implies that parECQ holds in $\mathcal{S}$, per the argument in Remark \ref{rem:parECQ-complementary}.
\end{rem}

\begin{lem}\label{lem:wkComp-poset}
    Suppose $\mathcal{S}=\langle\lang,\vdash\rangle$ is a logic. Let 
    \[
    \mathfrak{C}=\{(\Gamma,\Delta)\in\pow(\lang)\times\pow(\lang)\mid\,(\Gamma,\Delta)\hbox{ is weakly complementary, }\Gamma\neq\emptyset,\hbox{ and }\Gamma\cap\Delta=\emptyset\},
    \]
    and $\preceq\,\subseteq\mathfrak{C}\times\mathfrak{C}$ be a relation defined as follows.
    \[
    (\Gamma,\Delta)\preceq(\Gamma^\prime,\Delta^\prime)\quad\hbox{iff}\quad\Gamma\subseteq\Gamma^\prime\hbox{ and }\Delta\subseteq\Delta^\prime.
    \]
    Then, $(\mathfrak{C},\preceq)$ is a partially ordered set.
\end{lem}

\begin{proof}
    Straightforward.
\end{proof}

\begin{thm}\label{thm:parECQ+!F->>gECQ}
    Suppose $\mathcal{S}=\langle\lang,\vdash\rangle$ is a logic such that there does not exist any $\varphi\in\lang$ such that $C_\vdash(\varphi)=\lang$, and parECQ holds in it. Let 
    $(\mathfrak{C},\preceq)$ be the partially ordered set described in the above lemma. If gECQ holds in $\mathcal{S}$, then $(\mathfrak{C},\preceq)$ has a maximal element $(\Pi,\Lambda)$ such that $\Pi\cup\Lambda=\lang$.
\end{thm}

\begin{proof}
    We first note that since parECQ holds in $\mathcal{S}$, by Remark \ref{rem:parECQ-complementary}, there exists complementary, and hence, also weakly complementary $(\Gamma,\Delta)$ in $\mathcal{S}$ such that $\Gamma\times\Delta\neq\emptyset$. Then, since there is no $\varphi\in\lang$ with $C_\vdash(\varphi)=\lang$, by Lemma \ref{lem:!F<->disjoint.complementary}, every such complementary $(\Gamma,\Delta)$ in $\mathcal{S}$ is such that $\Gamma\cap\Delta=\emptyset$. Hence, $\mathfrak{C}\neq\emptyset$. Now, suppose gECQ holds in $\mathcal{S}$. 
    
    Let $\mathcal{C}$ be a chain in $(\mathfrak{C},\preceq)$. 
    We claim that $\left(\displaystyle\bigcup_{(\Gamma,\Delta)\in\mathcal{C}}\Gamma,\displaystyle\bigcup_{(\Gamma,\Delta)\in\mathcal{C}}\Delta\right)$ is an upper bound of $\mathcal{C}$ in $\mathfrak{C}$. 
    Since for every $(\Gamma,\Delta)\in\mathcal{C}\subseteq\mathfrak{C}$, $\Gamma\neq\emptyset$, $\displaystyle\bigcup_{(\Gamma,\Delta)\in\mathcal{C}}\Gamma\neq\emptyset$. Let $\alpha\in\displaystyle\bigcup_{(\Gamma,\Delta)\in\mathcal{C}}\Gamma$. This implies that there exists $(\Gamma,\Delta)\in\mathcal{C}$ such that $\alpha\in\Gamma$. Since $(\Gamma,\Delta)$ is weakly complementary, there exists $\beta\in\Delta\subseteq\displaystyle\bigcup_{(\Gamma,\Delta)\in\mathcal{C}}\Delta$ such that $C_\vdash(\{\alpha,\beta\})=\lang$. Thus, $\left(\displaystyle\bigcup_{(\Gamma,\Delta)\in\mathcal{C}}\Gamma,\displaystyle\bigcup_{(\Gamma,\Delta)\in\mathcal{C}}\Delta\right)$ is weakly complementary in $\mathcal{S}$. Next, to show that $\left(\displaystyle\bigcup_{(\Gamma,\Delta)\in\mathcal{C}}\Gamma\right)\bigcap\left(\displaystyle\bigcup_{(\Gamma,\Delta)\in\mathcal{C}}\Delta\right)=\emptyset$, suppose the contrary. Let $\alpha\in\left(\displaystyle\bigcup_{(\Gamma,\Delta)\in\mathcal{C}}\Gamma\right)\bigcap\left(\displaystyle\bigcup_{(\Gamma,\Delta)\in\mathcal{C}}\Delta\right)$. So, $\alpha\in\displaystyle\bigcup_{(\Gamma,\Delta)\in\mathcal{C}}\Gamma$, which implies that there exists $(\Gamma,\Delta)\in\mathcal{C}$ such that $\alpha\in\Gamma$. Similarly, since $\alpha\in\displaystyle\bigcup_{(\Gamma,\Delta)\in\mathcal{C}}\Delta$, there exists $(\Gamma^\prime,\Delta^\prime)\in\mathcal{C}$ such that $\alpha\in\Delta^\prime$. Now, since $\mathcal{C}$ is a chain, either $(\Gamma,\Delta)\preceq(\Gamma^\prime,\Delta^\prime)$ or $(\Gamma^\prime,\Delta^\prime)\preceq(\Gamma,\Delta)$. We can assume, without loss of generality, $(\Gamma,\Delta)\preceq(\Gamma^\prime,\Delta^\prime)$. Hence, $\Gamma\subseteq\Gamma^\prime$, which implies that $\alpha\in\Gamma^\prime$. Thus, $\alpha\in\Gamma^\prime\cap\Delta^\prime$. However, since $(\Gamma^\prime,\Delta^\prime)\in\mathcal{C}\subseteq\mathfrak{C}$, $\Gamma^\prime\cap\Delta^\prime=\emptyset$. This is a contradiction. Hence, $\left(\displaystyle\bigcup_{(\Gamma,\Delta)\in\mathcal{C}}\Gamma\right)\bigcap\left(\displaystyle\bigcup_{(\Gamma,\Delta)\in\mathcal{C}}\Delta\right)=\emptyset$. So, $\left(\displaystyle\bigcup_{(\Gamma,\Delta)\in\mathcal{C}}\Gamma,\displaystyle\bigcup_{(\Gamma,\Delta)\in\mathcal{C}}\Delta\right)\in\mathfrak{C}$. Clearly, $(\Gamma,\Delta)\preceq\left(\displaystyle\bigcup_{(\Gamma,\Delta)\in\mathcal{C}}\Gamma,\displaystyle\bigcup_{(\Gamma,\Delta)\in\mathcal{C}}\Delta\right)$ for each $(\Gamma,\Delta)\in\mathcal{C}$. Thus, $\left(\displaystyle\bigcup_{(\Gamma,\Delta)\in\mathcal{C}}\Gamma,\displaystyle\bigcup_{(\Gamma,\Delta)\in\mathcal{C}}\Delta\right)$ is an upper bound of $\mathcal{C}$ in $\mathfrak{C}$. Since $\mathcal{C}$ was an arbitrary chain in $(\mathfrak{C},\preceq)$, this proves that every chain in $(\mathfrak{C},\preceq)$ has an upper bound. Hence, by Zorn's lemma, we conclude that $(\mathfrak{C},\preceq)$ has a maximal element. 
    
    Let $(\Pi,\Lambda)\in\mathfrak{C}$ be such a maximal element. We claim that $\Pi\cup\Lambda=\lang$. Suppose the contrary. Then, there exists $\gamma\in\lang\setminus(\Pi\cup\Lambda)$. Since gECQ holds in $\mathcal{S}$, there exists $\delta\in\lang$ such that $C_\vdash(\{\gamma,\delta\})=\lang$. Moreover, since there does not exist any $\varphi\in\lang$ such that $C_\vdash(\varphi)=\lang$, $\delta\neq\gamma$. We fix such a $\delta$. There are now three possible cases.
    
    \textsc{Case 1:} $\delta\in\Lambda$. (Then, $\delta\notin\Pi$ since $\Pi\cap\Lambda=\emptyset$.) Let $\Pi^\prime=\Pi\cup\{\gamma\}$. Clearly, $\Pi^\prime\neq\emptyset$. Since $(\Pi,\Lambda)$ is weakly complementary and $C_{\vdash}(\{\gamma,\delta\})=\lang$, for every $\alpha\in\Pi^\prime$, there exists $\beta\in\Lambda$ such that $C_{\vdash}(\{\alpha,\beta\})=\lang$. So, $(\Pi^\prime,\Lambda)$ is weakly complementary. Finally, since $\Pi\cap\Lambda=\emptyset$ and $\gamma\notin\Pi\cup\Lambda$, $\Pi^\prime\cap\Lambda=\emptyset$. Thus, $(\Pi^\prime,\Lambda)\in\mathfrak{C}$. Now, since $\Pi\subsetneq\Pi^\prime$, $(\Pi,\Lambda)\preceq(\Pi^\prime,\Lambda)$ but $(\Pi,\Lambda)\neq(\Pi^\prime,\Lambda)$. This contradicts the maximality of $(\Pi,\Lambda)$ in $(\mathfrak{C},\preceq)$.

    \textsc{Case 2:} $\delta\in\Pi$. (Then, $\delta\notin\Lambda$, since $\Pi\cap\Lambda=\emptyset$.) We know that $\Pi\neq\emptyset$. Let $\Lambda^\prime=\Lambda\cup\{\gamma\}$. By the same arguments as in Case 1, $(\Pi,\Lambda^\prime)$ is weakly complementary, and $\Pi\cap\Lambda^\prime=\emptyset$. Thus, $(\Pi,\Lambda^\prime)\in\mathfrak{C}$. Again, since $\Lambda\subsetneq\Lambda^\prime$, $(\Pi,\Lambda)\preceq(\Pi,\Lambda^\prime)$ but $(\Pi,\Lambda)\neq(\Pi,\Lambda^\prime)$. This contradicts the maximality of $(\Pi,\Lambda)$ in $(\mathfrak{C},\preceq)$.
    
    \textsc{Case 3:} $\delta\in\lang\setminus(\Pi\cup\Lambda)$. Let $\Pi^\prime=\Pi\cup\{\gamma\}$ and $\Lambda^\prime=\Lambda\cup\{\delta\}$. Clearly, $\Pi^\prime\neq\emptyset$. By the same arguments as in the previous two cases, $(\Pi^\prime,\Lambda^\prime)$ is weakly complementary. Since $\Pi\cap\Lambda=\emptyset$, $\gamma,\delta\notin\Pi\cup\Lambda$, and $\delta\neq\gamma$, $\Pi^\prime\cap\Lambda^\prime=\emptyset$. Then, as $\Pi\subsetneq\Pi^\prime$ and $\Lambda\subsetneq\Lambda^\prime$, $(\Pi,\Lambda)\preceq(\Pi^\prime,\Lambda^\prime)$ but $(\Pi,\Lambda)\neq(\Pi^\prime,\Lambda^\prime)$. This contradicts the maximality of $(\Pi,\Lambda)$ in $(\mathfrak{C},\preceq)$.

    Since we arrive at a contradiction in each of the above three possible cases, we conclude that $\Pi\cup\Lambda=\lang$.
\end{proof}

\begin{rem}
    We note, from the above theorem, that, given a logic $\mathcal{S}=\langle\lang,\vdash\rangle$ that satisfies the conditions of the theorem, if gECQ holds, then $\lang$ can be partitioned into two disjoint pieces, viz., the sets $\Pi,\Lambda$ of a maximal element $(\Pi,\Lambda)\in\mathfrak{C}$ such that each member of $\Pi$ has a corresponding member in $\Lambda$, with which it explodes. Moreover, the maximality of $(\Pi,\Lambda)$, together with the fact that $\Pi\cap\Lambda=\emptyset$, implies that there does not exist $\alpha,\beta\in\Pi$ such that $C_\vdash(\{\alpha,\beta\})=\lang$. Hence, we can think of $\mathcal{S}$ as being `locally consistent,' as witnessed by $\Pi$.

    We can also conclude, from the above theorem, that if the partially ordered set $(\mathfrak{C},\preceq)$ has no such maximal element $(\Pi,\Lambda)$ such that $\Pi\cup\Lambda=\lang$, then gECQ fails in $\mathcal{S}$.
\end{rem}

We end this section with a concept of paraconsistency arising from the failure of a special case of parECQ in which $\Gamma=\lang$. This is defined below.

\begin{dfn}\label{dfn:K-paraconsistent}
    Suppose $\mathcal{S}=\langle\lang,\vdash\rangle$ is a logic, and $K\subseteq\lang$. $\mathcal{S}$ is said to be \emph{$K$-paraconsistent} if there exists $\alpha\in\lang$ such that $C_\vdash(\{\alpha,\beta\})\neq\lang$ for all $\beta\in K$.
\end{dfn}

The following are some easy observations that follow from the above definition.

\begin{thm}\label{thm:K-paraconsistent}
Suppose $\mathcal{S}=\langle\lang,\vdash\rangle$ is a logic.
    \begin{enumerate}[label=(\roman*)]
    \item $\mathcal{S}$ is $\emptyset$-paraconsistent.
    \item $\mathcal{S}$ is $\lang$-paraconsistent iff it is $NF$-paraconsistent.
    \item $\mathcal{S}$ is $K$-paraconsistent for some $K\subseteq\lang$ implies that it is $K^\prime$-paraconsistent for every $K^\prime\subseteq K$.
\end{enumerate}
\end{thm}

\begin{proof} These statements follow straightforwardly from the definition of $K$-paraconsistency.
\end{proof}

\begin{rem}
    The set $K\subseteq\lang$, in Definition \ref{dfn:K-paraconsistent}, can be varied to obtain different `degrees of paraconsistency,' somewhat like the degrees of unsolvability in Computability theory. More precisely, we can say that a logic $\mathcal{S}_1=\langle\lang,\vdash_1\rangle$ is more paraconsistent than a logic $\mathcal{S}_2=\langle\lang,\vdash_2\rangle$ if there exists a $K\subseteq\lang$ such that $\mathcal{S}_1$ is $K$-paraconsistent, but $\mathcal{S}_2$ is not. We plan to expand on this in future articles.
\end{rem}

\section{Conclusions}

In this paper, we have extended our earlier work on negation-free paraconsistency in \cite{BasuRoy2022}. The earlier article introduced the NF-paraconsistency as the failure of the generalized explosion principle gECQ fails. Several new examples of NF-paraconsistent logics have been provided here, the most notable of which are the ones involving variable inclusion logics, an area of research that is seeing a lot of interest lately. 

We have included the discussion on a quasi-negation, defined in \cite{BasuRoy2022} using the notion of NF-paraconsistency. This should be studied further, in general, and for specific logics. In the latter case, this can lead to characterizations of the sets of quasi-negations of a formula in these logics. In the case of logics with algebraic or topological interpretations, such as classical and intuitionistic propositional logics, quasi-negation can also be studied algebraically or topologically.

We have then focused on principles of explosion and, in the process, further generalized gECQ. Therefore, it is established that one can indeed talk about explosion in a broader sense without necessarily involving any logical connectives/operators. We have shown the interconnections, or the lack thereof, among the various principles of explosion. It is now imperative that we introduce a notion of paraconsistency via the failure of each of the distinct principles of explosion presented here. However, we have left this for future work. Further investigation is also required to find more examples of logics that satisfy, or fail, one or more of these new principles of explosion. In this connection, LFIs that are not finitely trivializable or not finitely gently explosive might be particularly interesting to study. As mentioned in the introductory section, we strongly believe that there are other principles of explosion possible. This needs to be investigated further.

We have introduced a couple of principles of partial explosion, which turned out to be equivalent. The goal with these is to localize consistency in a logic. Further work is needed, not only to find examples but also to see if other such principles are possible. The idea of $K$-paraconsistency was formulated as the failure of a particular case of the principles of partial explosion. As indicated at the end of the previous section, this needs to be rigorously studied, as well, to see if it can be used to compare paraconsistent logics.

Another interesting future project would be to connect the present study to Abstract Algebraic Logic (AAL). In this regard, we hope to find similarities between the present work with the research on inconsistency lemmas and antitheorems in AAL (discussed in \cite{Raftery2013}).

\section*{Acknowledgement}
The authors wish to express their gratitude to Prof.\ Mihir K.\ Chakraborty for his encouragement and advice during multiple discussions on earlier versions of the article.

\bibliographystyle{amsplain}
\bibliography{gexp}
\end{document}